\documentclass[12pt]{amsart}
\usepackage{amssymb, verbatim}
\usepackage{epsfig,color}
\usepackage{enumerate}
\usepackage{color,url}

\date{\today}

\newtheorem{theorem}{Theorem}[section]

\newtheorem{lemma}[theorem]{Lemma}

\theoremstyle{definition}

\theoremstyle{remark}

\makeatletter
\@namedef{subjclassname@2020}{%
 \textup{2020} Mathematics Subject Classification}
\makeatother

\numberwithin{equation}{section}

\usepackage[T1]{fontenc}

\newcommand{\hyper}[5]{\,{}_{#1}F_{#2}\left(\!\!%
\begin{array}{cc}{\displaystyle{#3}}\\[-0.1ex]
{\displaystyle{#4}} \end{array}\Big| \,{\displaystyle{#5}}
\right)}

\begin{document}

\title[Fourier transforms of some special functions]{Fourier transforms of
orthogonal polynomials on the cone}

\author[R. Akta\c{s} Karaman]{Rab\.{I}a  Akta\c{s} Karaman}
\address[R. Akta\c{s} Karaman]{Ankara University, Faculty of Science, Department of Mathematics, 06100, Tando\u{g}an, Ankara, Turkey}
\email[R. Akta\c{s} Karaman]{raktas@science.ankara.edu.tr}

\author[I. Area]{Iv\'an Area}
\address[I. Area]{IFCAE, Universidade de Vigo,
		Departamento de Matem\'atica Aplicada II,
              E. E. Aeron\'au\-tica e do Espazo,
              Campus As Lagoas s/n,
              32004 Ourense, Spain}
\email[I. Area]{area@uvigo.gal}


\begin{abstract}
The purpose of this paper is to obtain Fourier transforms of multivariate orthogonal polynomials on the cone such as Laguerre polynomials on the cone and Jacobi polynomials on the cone and to define two new families of multivariate orthogonal functions by using  Parseval's identity. Also, the obtained results are expressed in terms of the continuous Hahn polynomials.
\end{abstract}

\maketitle

\section{Introduction}

The study of orthogonal polynomials is rooted in classical mathematical analysis, with notable contributions from 19th-century mathematicians such as Lapace, Legendre, Jacobi, or  Chebyshev. The simplest example is the family of Legendre polynomials, which are usually introduced as orthogonal on the interval $[-1,1]$ with respect to a constant weight function. More generally, the weight function plays a critical role in defining the orthogonality condition, enabling the construction of families such as Chebyshev, Hermite, Laguerre and Jacobi (classical) polynomials, each suited to specific problems and applications. These univariate families have been extended in many directions. On the one hand, by considering another univariate weight functions, and on the other, by extending them to the multivariate situation(s). In the latter case, the domain of orthogonality plays a fundamental role.

\medskip
We shall focus this work in a combination of extension to the multivariate situation combined with certain integral transforms: fundamental mathematical techniques used to convert functions from one domain to another, often simplifying the analysis and solution of complex problems. By integrating a function with a kernel function, integral transforms reveal insights that might be less apparent in the function's original form. For instance, the very classical Fourier transform relates the time domain with the domain of frequency. These transforms have, therefore, interest from both theoretical and practical points of view. We just recall Fourier, Laplace, Beta, Hankel, Mellin and Whittaker transforms with various linked to special functions by considering appropriate kernels \cite{3}--\cite{21}.

\medskip
If we consider the simplest example of analyzing the Fourier transform of univariate orthogonal polynomials,  Hermite functions which are multiplied Hermite polynomials $H_{n}\left(  x\right)  $\ by $\exp \left(  -x^{2}/2\right)  $\ are eigenfunctions of Fourier transform \cite{7,8,9,19}. In \cite{9}, by the Fourier-Jacobi transform, it is investigated that classical Jacobi polynomials can be mapped onto Wilson polynomials. Also, Fourier transform of Jacobi polynomials and their close relation with continuous Hahn polynomials \cite{27} have been discussed by Koelink \cite{7}. The Fourier transforms of finite classical orthogonal polynomials $M_{n}^{\left(  p,q\right)  }\left( x\right)  $\ and $N_{n}^{\left(  p\right)  }\left(  x\right)  $ by Koepf and Masjed-Jamei \cite{8}, generalized ultraspherical and generalized Hermite polynomials and two symmetric sequences of finite orthogonal polynomials by Masjed-Jamei and Koepf \cite{11,21}, and Routh-Romanovski polynomials $J_{n}^{\left(  u,v\right)  }\left( x;a,b,c,d\right)  $ by Masjed-Jamei et al. \cite{20} have been studied.

\medskip
As for generalizations to the multivariate case of orthogonal polynomials, Tratnik \cite{Trat1, Trat2} has given
multivariable generalization both of all continuous and discrete families of the Askey tableau, providing hypergeometric representation, orthogonality weight function which applies with respect to subspaces of lower degree and biorthogonality within a given subspace. A non-trivial interaction for multivariable continuous Hahn polynomials has been presented by Koelink et al. \cite{Koe}. Moreover, in \cite{22,23,233,24} Fourier transforms of multivariate orthogonal polynomials and their applications have been investigated. Some families of orthogonal functions in terms of continuous Hahn polynomials have been obtained. In particular, in \cite{233} a new family of orthogonal functions have been derived by use of the Fourier transforms of the multivariate orthogonal polynomials on the unit ball and the Parseval identity.

\medskip
We would like also to pay special attention to orthogonal polynomials in a conic domain. In \cite{X20} Xu studied orthogonal polynomials and the Fourier orthogonal series on a cone in ${\mathbb{R}}^{d+1}$, showing that orthogonal polynomials on the cone are eigenfunctions of a second order partial differential operator. Later he presented the construction of semi-discrete localized tight frame in weighted $L^{2}$ norm and characterization of best approximation by polynomials on conic domains \cite{X21a}. More recently, Xu also studied orthogonal expansions with respect to the Laguerre type weight functions on the conic surface of revolution and the domain bounded by such a surface \cite{X21b}. We would like to remark that very recently two families of orthogonal polynomials in the cone have been studied in \cite{ABFX}.

\medskip
Motivated by some of these aforementioned investigations of the Fourier transforms of orthogonal polynomials, the main aim of this work is to study the Fourier transforms of the multivariate orthogonal polynomials on a conic domain.

\medskip
The paper is organized as follows. The next section is preliminary, where we recall the results of classical orthogonal polynomials on the unit ball and on the cone. In the third section, we obtain the Fourier transforms of the multivariate orthogonal polynomials on the cone ${\mathbb{V}}^{d+1}$ in terms of continuous Hahn polynomials. We first state the results for $d=1$ to illustrate the results and illuminate how the results on ${\mathbb{V}}^{d+1}$  are obtained, then we give the results on the cone ${\mathbb{V}}^{d+1}$ by induction. Finally, by use of Fourier transforms given here and Parseval's identity, new families of orthogonal functions are investigated.

\section{Preliminaries}\label{sec:bdn}
In this section, we state background materials on orthogonal polynomials that we shall need. We recall the basic results on the classical multivariate orthogonal polynomials on the unit ball and on the cone.

\subsection{Multivariate Orthogonal Polynomials}
In what follows we shall use multi-index notation (see \cite{29}):  ${\mathbf{k}}=(k_{1},\ldots,k_{d})\in{\mathbb{N}}_{0}^{d}$ and ${\boldsymbol{x}}=(x_{1},\ldots,x_{d})\in{\mathbb{R}}^{d}$.

Let ${\mathsf{w}}$ be a weight function on a domain $\Omega \subset{\mathbb{R}%
}^{d}$. Let ${\langle}\cdot,\cdot{\rangle}_{{\mathsf{w}}}$ be the inner
product defined by
\[
{\langle}f,g{\rangle}_{{\mathsf{w}}}=\int_{\Omega}f(\boldsymbol{x})g(\boldsymbol{x}){\mathsf{w}%
}(x)\mathrm{d}\boldsymbol{x}
\]
where $d\boldsymbol{x}=\mathrm{d}x_{1}\cdots \mathrm{d}x_{d}.$ Let $\Pi_{n}^{d}$ denote the space of polynomials of degree at most $n$ in $d$
variables. A polynomial~$P$ of degree $n$ is an orthogonal polynomial with
respect to the inner product~if
\[
{\langle}P,Q{\rangle}_{{\mathsf{w}}}=0,\qquad \forall Q\in \Pi_{n-1}^{d}.
\]
Assume that ${\mathsf{w}}$ is regular so that orthogonal polynomials with
respect to the inner product are well-defined. Let ${\mathcal{V}}_{n}%
^{d}(\Omega,{\mathsf{w}})$ be the space of orthogonal polynomials of degree
$n$ in $d$ variables with respect to the inner product. It is known that (see \cite{29})
\[
r_{n}^{d}:=\dim{\mathcal{V}}_{n}^{d}(\Omega,{\mathsf{w}})=\binom{n+d-1}%
{n},\qquad n=0,1,2,\ldots.
\]
For $d>1$, the space ${\mathcal{V}}_{n}^{d}(\Omega,{\mathsf{w}})$ contains
infinitely many bases. Moreover, since orthogonality is defined as orthogonal
to all polynomials of lower degrees, elements of a basis may not be mutually
orthogonal. A basis $\{P_{j}^{n}:1\leq j\leq r_{n}^{d}\}$ of ${\mathcal{V}%
}_{n}^{d}(\Omega,{\mathsf{w}})$ is called an orthogonal basis if ${\langle
}P_{j}^{n},P_{k}^{n}{\rangle}_{{\mathsf{w}}}=0$ whenever $j\neq k$.

\subsection{Orthogonal polynomials on the unit ball}
Let  $\left \Vert \boldsymbol{x}\right \Vert :=\left(  x_{1}^{2}+\cdots+x_{d}^{2}\right)  ^{1/2}$ for $\boldsymbol{x=}\left(  x_{1},\dots,x_{d}\right)  \in \mathbb{R}^{d}$. For $\mu>-\frac{1}{2}$, the classical weight function on the unit ball
${\mathbb{B}}^{d}=\{\boldsymbol{x}\in{\mathbb{R}}^{d}:\Vert \boldsymbol{x}\Vert \leq1\}$ of ${\mathbb{R}^{d}
}$ is (see \cite{29})
\[
{\mathsf{w}}_{\mu}(\boldsymbol{x})=(1-\Vert \boldsymbol{x}\Vert^{2})^{\mu-\frac{1}{2}},\quad \mu
>-\tfrac{1}{2},\quad \boldsymbol{x}\in{\mathbb{B}}^{d}.
\]

When $d=1$, the ball is $[-1,1]$ and the associated orthogonal polynomials are
the classical Gegenbauer polynomials $C_{n}^{\mu}$ and they are defined by hypergeometric function
$_{2}F_{1}$ as \cite[p. 277, Eq. (4)]{25}%
\begin{equation}\label{hyper}
C_{n}^{\left(  \mu \right)  }\left(  x\right)  =\frac{\left(
2\mu\right)  _{n}}{n!} \,  \hyper{2}{1}{-n,n+2\mu}{\mu+\frac{1}{2}}{\frac{1-x}{2}},
\end{equation}
where the Gauss hypergeometric function $_{2}F_{1}$ is the special case $p=2,~q=1$ of the $\left(  p,q\right)  $ generalized hypergeometric function
(cf. \cite{26})%
\begin{equation}
\hyper{p}{q}{a_{1},\ a_{2},\dots,a_{p}}{b_{1},\ b_{2},\dots,b_{q}}{x}=
{\displaystyle \sum \limits_{n=0}^{\infty}}
\frac{\left(  a_{1}\right)  _{n}\left(  a_{2}\right)  _{n} \dots \left(
a_{p}\right)  _{n}}{\left(  b_{1}\right)  _{n}\left(  b_{2}\right)
_{n}\dots \left(  b_{q}\right)  _{n}}\frac{x^{n}}{n!}\label{genhyper}%
\end{equation}
where  $\left(  \alpha \right)  _{n}=\alpha \left(  \alpha+1\right) \cdots. \left(
\alpha+n-1\right)  ~,~n\geq1$ , $\left(  \alpha \right)  _{0}=1$ denotes the
Pochhamber symbol.
The Gegenbauer polynomials satisfy the following orthogonality relation \cite[p.281, Eq. (28)]{25}%
\begin{equation}
 \int \limits_{-1}^{1}\left(  1-x^{2}\right)  ^{\mu-\frac{1}{2}}%
C_{n}^{\left(  \mu \right)  }\left(  x\right)  C_{m}^{\left(
\mu \right)  }\left(  x\right)  dx=h_{n}^{\mu}~\delta_{n,m}%
,\label{ort} \quad
 \left(  m,n\in \mathbb{N}_{0}:= \mathbb{N}. \cup \left \{  0\right \}  \right),
\end{equation}
where $h_{n}^{\mu}$ is given by%
\begin{equation}
h_{n}^{\mu}=\frac{\left(  2\mu \right)  _{n}\Gamma \left(  \mu
+\frac{1}{2}\right)  \Gamma \left(  \frac{1}{2}\right)  }{n!\left(
n+\mu \right)  \Gamma \left( \mu \right)  }\label{gnorm}%
\end{equation}
and $\delta_{n,m}$ is the Kronecker delta, and the Gamma function $\Gamma \left(  x\right)$ is defined by (cf. \cite{26})%
\begin{equation}
\Gamma \left(  x\right)  =\int \limits_{0}^{\infty}t^{x-1}e^{-t}dt,\  \Re \left(
x\right)  >0.\label{2}%
\end{equation}.

For $d>1$, the orthogonal polynomials with respect to ${\mathsf{w}}_{\mu}$ on
${\mathbb{B}}^{d}$ are well studied; see~\cite[Section 5.2]{29}. Let
${\mathcal{V}}_{n}^{d}({\mathbb{B}}^{d},{\mathsf{w}}_{\mu})$ be the space of OPs
for this weight function. There are many distinct bases for this space. We are
mostly interested in the Jacobi basis.
 The elements of the space ${\mathcal{V}}_{n}^{d}({\mathbb{B}}^{d},{\mathsf{w}}_{\mu})$ are eigenfunctions
of a second order partial differential equation \cite[p.141, Eq. (5.2.3)]{29}%
\begin{equation*}
 \sum \limits_{i=1}^{d}\frac{\partial^{2}P}{\partial x_{i}^{2}}-\sum
\limits_{j=1}^{d}\frac{\partial}{\partial x_{j}}x_{j}\left[  2\mu
-1+\sum \limits_{i=1}^{d}x_{i}\frac{\partial}{\partial x_{i}}\right]  P
 =-\left(  n+d\right)  \left(  n+2\mu-1\right)  P.
\end{equation*}
The space $\mathcal{V}_{n}^{d}$ has several different bases. One of orthogonal
bases of the space $\mathcal{V}_{n}^{d}$ is in terms of the Gegenbauer
polynomials $C_{n_{j}}^{\left(  \lambda_{j}\right)  }\left(  x_{j}\right)  $
as follows \cite[p.143]{29}%
\begin{equation}
P_{\mathbf{n}}^{\mu}\left(  \boldsymbol{x}\right)  =\prod \limits_{j=1}%
^{d}\left(  1-\left \Vert \boldsymbol{x}_{j-1}\right \Vert ^{2}\right)
^{\frac{n_{j}}{2}}C_{n_{j}}^{\left(  \lambda_{j}\right)  }\left(  \frac{x_{j}%
}{\sqrt{1-\left \Vert \boldsymbol{x}_{j-1}\right \Vert ^{2}}}\right) \label{P}%
\end{equation}
where $\lambda_{j}=\mu+\left \vert \mathbf{n}^{j+1}\right \vert +\frac{d-j}{2},$
and
\begin{equation}\label{notation}
\begin{cases}
\boldsymbol{x}_{0}   =0,\  \  \boldsymbol{x}_{j}=\left(  x_{1},\dots,x_{j}\right)  , \\
\mathbf{n}   =\left(  n_{1},\dots,n_{d}\right)  ,\  \text{\ }\left \vert
\mathbf{n}\right \vert =n_{1}+\cdots+n_{d}=n,\\
\mathbf{n}^{j}   =\left(  n_{j},\dots,n_{d}\right)  ,\text{ \ }\left \vert
\mathbf{n}^{j}\right \vert =n_{j}+\cdots+n_{d},\  \text{\  \ }1\leq j\leq d,
\end{cases}
\end{equation}
and $\mathbf{n}^{d+1}:=0.$ More precisely,%
\[
\int \limits_{\mathbb{B}^{d}}{\mathsf{w}}_{\mu}\left(  \boldsymbol{x}\right)
P_{\mathbf{n}}^{\mu}\left(  \boldsymbol{x}\right)  P_{\mathbf{m}}^{\mu}\left(
\boldsymbol{x}\right)  d\boldsymbol{x}=h_{\mathbf{n}}^{\mu}\delta_{\mathbf{n}%
,\mathbf{m}}
\]
where $\delta_{\mathbf{n},\mathbf{m}}=\delta_{n_{1},m_{1}}\cdots \delta
_{n_{d},m_{d}}$ and $h_{\mathbf{n}}^{\mu}$ is given by (cf. \cite{29})%
\begin{equation}
h_{\mathbf{n}}^{\mu}=\frac{\pi^{d/2}\Gamma \left(  \mu+\frac{1}{2}\right)
\left(  \mu+\frac{d}{2}\right)  _{\left \vert \mathbf{n}\right \vert }}%
{\Gamma \left(  \mu+\frac{d+1}{2}+\left \vert \mathbf{n}\right \vert \right)
}\prod \limits_{j=1}^{d}\frac{\left(  \mu+\frac{d-j}{2}\right)  _{\left \vert
\mathbf{n}^{j}\right \vert }\left(  2\mu+2\left \vert \mathbf{n}^{j+1}%
\right \vert +d-j\right)  _{n_{j}}}{n_{j}!\left(  \mu+\frac{d-j+1}{2}\right)
_{\left \vert \mathbf{n}^{j}\right \vert }}.\label{Norm}%
\end{equation}

\subsection{Orthogonal polynomials on the cone}

We recall orthogonal polynomials on the conic domain
\begin{equation}\label{eq:conedomain}
{\mathbb{V}}^{d+1}=\{(t,\boldsymbol{x})\in{\mathbb{R}}^{d+1}:\, \Vert \boldsymbol{x}\Vert \leq
t,\,\boldsymbol{x}\in{\mathbb{R}}^{d},\,t\geq0\},
\end{equation}
on ${\mathbb{R}}^{d+1}$ and the orthogonality is defined with respect to the weight function \cite{X20}
\[
W(t,\boldsymbol{x})=w(t)(t^{2}-\Vert \boldsymbol{x}\Vert^{2})^{\mu-\frac{1}{2}},\qquad \mu>-\tfrac{1}{2},
\]
where $w$ is a weight function defined on ${\mathbb{R}}_{+}=[0,\infty)$ and
the cone could be regarded as finite if $w$ is supported on the finite
interval, say, $[0,1]$. Orthogonal polynomials on the cone have only been studied in recent years. Two of the most significant cases are the Laguerre polynomials on the cone, with the weight function $w(t)=t^{{\beta}}{\mathrm{e}}^{-t}$,  and the Jacobi polynomials on the cone, with the weight function $w(t)=t^{{\beta}}(1-t)^{{\gamma}}$. For these cases, an orthonormal basis is defined in \cite{X20}, and it is shown that the orthogonal polynomials in ${\mathcal{V}}_{n}^{d+1}$ are eigenfunctions of a second-order differential operator. Additionally, they satisfy an addition formula, which provides a closed-form expression for the reproducing kernel of the space ${\mathcal{V}}%
_{n}^{d+1}$. This kernel is crucial for advancing approximation theory and computational harmonic analysis over the cone, as discussed in \cite{X21a, X21b}. Furthermore, the monomial basis and the basis defined by the Rodrigues formulas on the cone are studied in \cite{ABFX}.

Let $w$ be a weight function on an interval in ${\mathbb{R}}$. We can assume
that the interval is either $[0,1]$ or ${\mathbb{R}}_{+}=[0,\infty)$ without
losing generality. We define the weight function
\begin{equation}
W_{\mu}(t,\boldsymbol{x})=w(t)(t^{2}-\Vert \boldsymbol{x}\Vert^{2})^{\mu-\frac{1}{2}},\quad \mu
>-\tfrac{1}{2}\label{eq:Wmu-cone}%
\end{equation}
and define the inner product
\[
{\langle}f,g{\rangle}_{\mu}:=\int_{{\mathbb{V}}^{d+1}}f(t,\boldsymbol{x})g(t,\boldsymbol{x})W_{\mu
}(t,\boldsymbol{x})\mathrm{d}\boldsymbol{x}\mathrm{d}t
\]
on the cone ${\mathbb{V}}^{d+1}$. By using the separation of variables, the following equality is satisfied
\[
\int_{{\mathbb{V}}^{d+1}}f(t,\boldsymbol{x})W_{\mu}(t,\boldsymbol{x})\mathrm{d}\boldsymbol{x}\mathrm{d}t=\int
_{0}^{\infty}\int_{{\mathbb{B}}^{d}}f(t,t\boldsymbol{y})(1-\Vert \boldsymbol{y}\Vert^{2})^{\mu-\frac
{1}{2}}\mathrm{d}\boldsymbol{y}\,t^{d+2\mu-1}w(t)\mathrm{d}t.
\]
Let ${\mathcal{V}}_{n}({\mathbb{V}}^{d+1},W_{\mu})$ be the space of orthogonal
polynomials of degree $n$ in $d+1$ variables with respect to the inner product
${\langle}f,g{\rangle}_{\mu}$. A basis of this space can be presented in terms of
orthogonal polynomials on the unit ball and a family of orthogonal polynomials
in one variable as in the following lemma \cite{ABFX,X20}.

\begin{lemma}
\label{prop:OPcone} Let ${\mathbb{P}}_{m} = \{P_{\mathbf{k}}: |{\mathbf{k}}| =
m\}$ be a basis of ${\mathcal{V}}_{m}({\mathbb{B}}^{d}, {\mathsf{w}}_{\mu})$
for $m \le n$ and let $q_{n-m}^{\alpha}$ be an orthogonal polynomial in one
variable with respect to the weight function $t^{\alpha} w(t)$ on
${\mathbb{R}}_{+}$. Define
\begin{equation}
\label{eq:sQcone}{\mathsf{Q}}_{{\mathbf{k}},n} (t,\boldsymbol{x}) = q_{n-m}^{{\alpha}_{m}%
}(t) t^{m} P_{\mathbf{k}}\! \left( \frac{\boldsymbol{x}}{t}\right) , \qquad|{\mathbf{k}}|
=m, \quad0 \le m \le n,
\end{equation}
where ${\alpha}_{m} = d+2m+2 \mu-1$. Then ${\mathbb{Q}}_{n} = \{ {\mathsf{Q}%
}_{{\mathbf{k}},n}: |{\mathbf{k}}| =m, \, 0 \le m \le n\}$ is a basis of
${\mathcal{V}}_{n}({\mathbb{V}}^{d+1}, W_{\mu})$. In particular, if
${\mathbb{P}}_{m} $ is an orthogonal basis for ${\mathcal{V}}_{m}({\mathbb{B}%
}^{d}, {\mathsf{w}}_{\mu})$ for $0 \le m \le n$, then ${\mathbb{Q}}_{n}$ is an
orthogonal basis for ${\mathcal{V}}_{n}({\mathbb{V}}^{d+1}, W_{\mu})$.
\end{lemma}

The most notable cases occur when $w$ is a classical weight function, allowing $q_{n-m}^{{\alpha}_{m}}$ to be explicitly expressed in terms of classical orthogonal polynomials. This leads to two families of classical weight functions, which give rise to the Laguerre polynomials on the cone and the Jacobi polynomials on the cone, as studied in  \cite{X20}. These two families of polynomials are given as follows.

\subsubsection{Laguerre polynomials on the cone}

In this case, the cone \eqref{eq:conedomain} is  unbounded with $t \in [0,+\infty)$ and the weight function $w(t)=t^{{\beta}}{\mathrm{e}}^{-t}$, so that $W_{\mu}$
in \eqref{eq:Wmu-cone} becomes
\[
W_{{\beta},\mu}(t,\boldsymbol{x})=(t^{2}-\Vert \boldsymbol{x}\Vert^{2})^{\mu-\frac{1}{2}}t^{{\beta}%
}{\mathrm{e}}^{-t},\quad \mu>-\tfrac{1}{2},\quad{\beta}>-d,
\]

The Laguerre polynomial $L_{n}^{{\alpha}}$ defined by, for
${\alpha}>-1$,
\begin{equation}
L_{n}^{{\alpha}}(t)=\frac{({\alpha}+1)_{n}}{n!} \hyper{1}{1}{-n}{\alpha+1}{t} =\frac{({\alpha}+1)_{n}}{n!}\sum_{k=0}^{n}\frac{(-n)_{k}}{({\alpha
}+1)_{k}k!}t^{k}\label{102}
\end{equation}
satisfies the orthogonal relation \cite{25}
\begin{align}
\int_{0}^{\infty}L_{n}^{{\alpha}}(t)L_{m}^{{\alpha}}(t)t^{{\alpha}}%
{\mathrm{e}}^{-t}\mathrm{d}t=\frac{\Gamma({\alpha}+n+1)}{n!}\delta_{m,n}\label{ort-L}.
\end{align}
The orthogonal polynomials ${\mathsf{Q}}_{{\mathbf{k}},n}$ in
\eqref{eq:sQcone} on the cone are now presented in terms of the Laguerre
polynomials. They are denoted by ${\mathsf{L}}_{{\mathbf{k}},n}^{\beta,\mu}$ and they are given as
\begin{equation}\label{eq:laguerrecone}
{\mathsf{L}}_{{\mathbf{k}},n}^{\beta,\mu}(t,\boldsymbol{x})=L_{n-m}^{2m+2\mu+{\beta}+d-1}%
(t)t^{m}P_{{\mathbf{k}}}^{m}\left(  \frac{\boldsymbol{x}}{t}\right)  ,\quad|{\mathbf{k}%
}|=m,\, \,0\leq m\leq n.
\end{equation}

\subsubsection{Jacobi polynomials on the cone}

In this case, the cone \eqref{eq:conedomain} is bounded, without loss of generality $t \in [0,1]$, and the weight function $w(t)=t^{{\beta}}(1-t)^{{\gamma}}$ with ${\beta}>-d$ and ${\gamma}>-1$. Now $W_{\mu}$ in~\eqref{eq:Wmu-cone} becomes
\[
W_{{\beta},{\gamma},\mu}(t,\boldsymbol{x})=(t^{2}-\Vert \boldsymbol{x}\Vert^{2})^{\mu-\frac{1}{2}%
}t^{{\beta}}(1-t)^{{\gamma}},\quad \mu>-\tfrac{1}{2},\, \,{\gamma}>-1,\,
\,{\beta}>-d.
\]

The Jacobi polynomial $P_{n}^{({\alpha},{\beta})}$ is defined by,
for ${\alpha},{\beta}>-1$,
\begin{align}
P_{n}^{{\alpha},{\beta}}(t)=\frac{({\alpha}+1)_{n}}{n!} \hyper{2}{1}{-n,n+{\alpha}+{\beta}+1}{{\alpha}+1}{\frac{1-t}{2}},
\end{align}\label{200}
in terms of the hypergeometric function and it satisfies the orthogonal relation \cite{25}
\begin{align}
\int_{-1}^{1}P_{n}^{({\alpha},{\beta})}(t)P_{m}^{({\alpha},{\beta}%
)}(t)(1-t)^{{\alpha}}(1+t)^{{\beta}}\mathrm{d}t=h_{n}^{({\alpha},{\beta}%
)}\delta_{n,m} \label{ort-J},
\end{align}
where
\[
h_{n}^{({\alpha},{\beta})}=\frac{2^{\alpha+\beta+1}\Gamma \left(  n+\alpha+1\right)  \Gamma \left(
n+\beta+1\right)  }{\left(  2n+\alpha+\beta+1\right)  \Gamma \left(
n+\alpha+\beta+1\right)  n!}.
\]
The orthogonal polynomials ${\mathsf{Q}}_{{\mathbf{k}},n}$ in
\eqref{eq:sQcone} on the cone are presented in terms of the Jacobi polynomials.
They are denoted by ${\mathsf{J}}_{{\mathbf{k}},n}^{\beta,\mu,\gamma}$ and they are given as
\begin{equation}\label{eq:jacobicone}
{\mathsf{J}}_{{\mathbf{k}},n}^{\beta,\mu,\gamma}(t,\boldsymbol{x})=P_{n-m}^{(2m+2\mu+{\beta}+d-1,{\gamma}%
)}(1-2t)t^{m}P_{{\mathbf{k}}}^{m}\left(  \frac{\boldsymbol{x}}{t}\right)  ,\quad
|{\mathbf{k}}|=m,\, \,0\leq m\leq n.
\end{equation}

\section{The Fourier Transforms of Orthogonal Polynomials on the Cone}\label{sec:ft}

The main results of this investigation are to find the Fourier transformations of the classical orthogonal polynomials on the cone ${\mathbb{V}}^{d+1}$ and, to obtain new families of multivariate orthogonal functions in terms of multivariable Hahn polynomials.

The univariate Fourier transform for a function $f(x)$ is defined by \cite[p.111, Eq. (7.1)]{3}%
\begin{equation}%
\mathcal{F}
\left(  f\left(  x\right)  \right)  =\int \limits_{-\infty}^{\infty}e^{-i\xi
x}f\left(  x\right)  dx \label{16}%
\end{equation}
and in $d$-variable case, the Fourier transform for a function $f(x_{1},\dots,x_{d})$ is defined by \cite[p. 182, Eq. (11.1a)]{3}%
\begin{equation}
\mathcal{F}%
\left(  f\left(  x_{1},\dots,x_{d}\right)  \right)  =\int \limits_{-\infty
}^{\infty}\cdots \int \limits_{-\infty}^{\infty}e^{-i\left(  \xi_{1}x_{1}%
+\cdots +\xi_{d}x_{d}\right)  }f\left(  x_{1},\dots,x_{d}\right)  dx_{1}\cdots dx_{d}.
\label{17}%
\end{equation}
For our purposes, we first recall the results for Fourier transform of orthogonal polynomials on the unit ball given in \cite{233}.

Let the function be given in terms of orthogonal polynomials on the ball as \cite{233}
\begin{multline}
f_{d}\left( \boldsymbol{x};\mathbf{k},a,\mu \right)  :=f_{d}\left(x_{1},\dots,x_{d};k_{1},\dots,k_{d},a,\mu \right)  \\
 =\prod \limits_{j=1}^{d}\left(  1-\tanh^{2}x_{j}\right)  ^{a+\frac{d-j}{4}}\ P_{\mathbf{k}}^{\mu}\left(  \upsilon_{1},\dots,\upsilon_{d}\right),\label{150}%
\end{multline}
for $d\geq1$, where $a,\mu$ are real parameters and
\begin{align*}
\upsilon_{1}\left(  x_{1}\right)   & =\upsilon_{1}=\tanh x_{1},\\
\upsilon_{d}\left(  x_{1},\dots,x_{d}\right)   & =\upsilon_{d}=\tanh x_{d}%
\sqrt{\left(  1-\tanh^{2}x_{1}\right)  \left(  1-\tanh^{2}x_{2}\right)
\dots \left(  1-\tanh^{2}x_{d-1}\right)  },
\end{align*}
for $d\geq2$. From the latter expression, the function $f_{d}$ can be written in terms of $f_{d-1}$ in the following forms \cite{233}
\begin{multline}
 f_{d}\left(  x_{1},\dots,x_{d};k_{1},\dots,k_{d},a,\mu \right)  \\
 =\left(  1-\tanh^{2}x_{1}\right)  ^{a+\frac{k_{2}+\cdots+k_{d}}{2}+\frac{d-1}{4}}C_{k_{1}}^{\left(  k_{2}+\cdots+k_{d}+\mu+\frac{d-1}{2}\right)  }\left(\tanh x_{1}\right)  \\
 \times f_{d-1}\left(  x_{2},\dots,x_{d};k_{2},\dots,k_{d},a,\mu \right),\label{g1}%
\end{multline}
or%
\begin{multline}
f_{d}\left(  x_{1},\dots,x_{d};k_{1},\dots,k_{d},a,\mu \right)
 =\left(  1-\tanh^{2}x_{d}\right)  ^{a}C_{k_{d}}^{\left(  \mu \right)
}\left(  \tanh x_{d}\right)  \\
 \times f_{d-1}\left(  x_{1},\dots,x_{d-1};k_{1},\dots,k_{d-1},a+\frac{k_{d}}%
{2}+\frac{1}{4},\mu+k_{d}+\frac{1}{2}\right), \label{g2}%
\end{multline}
valid for $d\geq1$.

The Fourier transform of $ f_{d}\left(   \boldsymbol{x};\mathbf{k},a,\mu \right)$ defined in \eqref{150} has been calculated by consecutively iteration in \cite{233}.
\begin{lemma}\cite{233}
\label{prop:OPcone2}
The Fourier transform of the function $f_{d}\left( \boldsymbol{x};\mathbf{k},a,\mu \right)  $ is explicitly given by
\begin{multline}%
\mathcal{F} \left(  f_{d}\left( \boldsymbol{x};\mathbf{k},a,\mu \right)  \right)  :=
\mathcal{F}
\left(  f_{d}\left(  x_{1},\dots,x_{d};k_{1},\dots,k_{d},a,\mu \right)  \right) \\
=2^{2da+\frac{d\left(  d-5\right)  }{4}+\text{$
{\textstyle \sum \limits_{j=1}^{d-1}}
$}{jk}_{j+1}}\prod \limits_{j=1}^{d}\left \{  \frac{\left(  2\left(
\left \vert \mathbf{k}^{j+1}\right \vert +\mu+\frac{d-j}{2}\right)  \right)
_{k_{j}}}{k_{j}!}\Theta_{j}^{d}\left(  a,\mu,\mathbf{k};\xi_{j}\right)
\right \}  ,  \label{18}%
\end{multline}
where%
\begin{multline*}
\Theta_{j}^{d}\left(  a,\mu,\mathbf{k};\xi_{j}\right)  =B\left(a+\frac{\left \vert \mathbf{k}^{j+1}\right \vert +i\xi_{j}}{2}+\frac{d-j}{4},a+\frac{\left \vert \mathbf{k}^{j+1}\right \vert -i\xi_{j}}{2}+\frac{d-j}{4}\right)    \\
\times
\hyper{3}{2}{-k_{j},\ k_{j}+2\left(  \left \vert \mathbf{k}^{j+1}\right \vert +\mu+\frac{d-j}{2}\right)  ,\ a+\frac{\left \vert \mathbf{k}^{j+1}\right \vert +i\xi_{j}}{2}+\frac{d-j}{4}}{\left \vert \mathbf{k}^{j+1}\right \vert +\mu+\frac{d-j+1}{2},\  \left \vert \mathbf{k}^{j+1}\right \vert +2a+\frac{d-j}{2}}{1},
\end{multline*}
or in terms of the continuous Hahn polynomials
\begin{multline*}
\Theta_{j}^{d}\left(  a,\mu,\mathbf{k};\xi_{j}\right)  =\frac{k_{j}!}{i^{k_{j}}\left(  \left \vert \mathbf{k}^{j+1}\right \vert +\mu+\frac{d-j+1}{2}\right)  _{k_{j}}\left(  \left \vert \mathbf{k}^{j+1}\right \vert +2a+\frac{d-j}{2}\right)  _{k_{j}}}   \\
\times B\left(  a+\frac{\left \vert \mathbf{k}^{j+1}\right \vert +i\xi_{j}}{2}+\frac{d-j}{4},a+\frac{\left \vert \mathbf{k}^{j+1}\right \vert -i\xi_{j}}{2}+\frac{d-j}{4}\right)    \\
\times p_{k_{j}}\left(  \frac{\xi_{j}}{2};a+\frac{\left \vert \mathbf{k}^{j+1}\right \vert }{2}+\frac{d-j}{4},\mu-a+\frac{\left \vert \mathbf{k}^{j+1}\right \vert +1}{2}+\frac{d-j}{4}\right.    \\
\left.  ,\mu-a+\frac{\left \vert \mathbf{k}^{j+1}\right \vert +1}{2}+\frac{d-j}{4},a+\frac{\left \vert \mathbf{k}^{j+1}\right \vert }{2}+\frac{d-j}{4}\right),
\end{multline*}
where the beta function is given by (cf. \cite{26})
\begin{equation}
B\left(  a,b\right)  ={\displaystyle \int \limits_{0}^{1}}x^{a-1}\left(
 1-x\right)  ^{b-1}dx=\frac{\Gamma \left(  a\right)\Gamma \left(  b\right)  }{\Gamma \left(  a+b\right)  }, \qquad \Re\left(  a\right)  ,\Re\left(  b\right)  >0,
\end{equation}
and the continuous Hahn polynomials are defined as \cite{27}
\begin{equation}
p_{k}\left( x;a,b,c,d\right)
=i^{k}\frac{\left( a+c\right) _{k}\left(a+d\right) _{k}}{k!} \hyper{3}{2}{-k,\ k+a+b+c+d-1,\ a+ix}{a+c,\ a+d}{1} \label{hahn}
\end{equation}
in terms of the $_{3}F_{2}$ hypergeometric function.
\end{lemma}

\subsection{The Fourier transform of Laguerre polynomials on the cone}
Let us consider a function in terms of Laguerre polynomials on the cone \eqref{eq:laguerrecone} as
\begin{align}
g_{n,\mathbf{k}}\left( t,\boldsymbol{x};a,b,\beta,\mu \right)   & :=g_{n,k_{1},\dots,k_{d}}\left(
t,x_{1},\dots,x_{d};a,b,\beta,\mu  \right) \nonumber \\
& =\prod \limits_{j=1}^{d}\left(  1-\tanh^{2}x_{j}\right)  ^{a+\frac{d-j}{4}%
}{\mathsf{L}}_{{\mathbf{k}},n}^{\beta,\mu}\left(  \tau_{1},\dots,\tau_{d},\tau_{d+1}\right)\exp(-\frac{e^{t}}{2}+bt)
,\label{15}%
\end{align}
for $d\geq1$, where $a,b,\beta,\mu$ real parameters and
\begin{align*}
\tau_{1}\left(  t\right)   & =\tau_{1}=e^t,\\
\tau_{2}\left( t, x_{1}\right)   & =\tau_{2}=e^t\tanh x_{1},\\
\tau_{d+1}\left( t, x_{1},\dots,x_{d}\right)   & =\tau_{d+1}=e^t\tanh x_{d} \\
& \times \sqrt{\left(  1-\tanh^{2}x_{1}\right)  \left(  1-\tanh^{2}x_{2}\right) \cdots \left(  1-\tanh^{2}x_{d-1}\right)  },
\end{align*}
for $d\geq1$.
In explicit form, we can write
\begin{align*}
g_{n,\mathbf{k}}\left( t,\boldsymbol{x};a,b,\beta,\mu \right) & =\prod \limits_{j=1}^{d}\left(  1-\tanh^{2}x_{j}\right)  ^{a+\frac{d-j}{4}%
}\ L_{n-\left \vert \mathbf{k}\right \vert }^{2\left \vert \mathbf{k}\right \vert+2\mu+\beta+d-1}\left(e^t\right) \\
&\times \exp(-\frac{e^{t}}{2}+bt+\left \vert \mathbf{k}\right \vert t)\prod \limits_{j=1}^{d-1}\left(  1-\tanh^{2}x_{j}\right)  ^\frac{k_{j+1}+\cdots +k_{d}}{2}
\prod \limits_{j=1}^{d}C_{k_{j}}^{\left(  \lambda_{j} \right)}\left(\tanh x_{j}\right)
\end{align*}
where $\lambda_{j}=\mu+\left \vert \mathbf{k}^{j+1}\right \vert +\frac{d-j}{2}$, or
\begin{equation}
g_{n,\mathbf{k}}\left( t, \boldsymbol{x};a,b,\beta,\mu \right) \\
=\exp(-\frac{e^{t}}{2}+bt+\left \vert \mathbf{k}\right \vert t) L_{n-\left \vert \mathbf{k}\right \vert }^{2\left \vert \mathbf{k}\right \vert+2\mu+\beta+d-1}\left(e^t\right)f_{d}\left(
\boldsymbol{x};\mathbf{k},a,\mu \right)\label{100}
\end{equation}
where $f_{d}\left(\boldsymbol{x};\mathbf{k},a,\mu \right)$ is defined in \eqref{150}.

We now calculate the Fourier transform of the function $g_{n,\mathbf{k}}\left( t,\boldsymbol{x};a,b,\beta,\mu \right)$ given by (\ref{100}) by using Lemma \ref{prop:OPcone2}.

\begin{theorem}
The Fourier transform of the function $g_{n,\mathbf{k}}\left( t, \boldsymbol{x};a,b,\beta,\mu \right)$ is given explicitly as follows%
\begin{multline}%
\mathcal{F}
\left(  g_{n,\mathbf{k}}\left(  t,\boldsymbol{x};a,b,\beta,\mu \right)  \right)  :=%
\mathcal{F}\left(  g_{n,\mathbf{k}}\left( t, x_{1},\dots,x_{d};a,b,\beta,\mu \right)  \right) \\
=2^{b+\left \vert \mathbf{k}\right \vert -i\xi_{d+1}+2ad+\frac{d\left(  d-5\right)  }{4}+\text{$
{\textstyle \sum \limits_{j=1}^{d-1}}
$}{ jk}_{j+1}}\frac{ \left(  2 \left\vert \mathbf{k} \right\vert + 2\mu + \beta + d \right)_{n - \left\vert \mathbf{k} \right\vert} \Gamma \left( b + \left\vert \mathbf{k} \right\vert - i \xi_{d+1} \right)}{\left( n - \left\vert \mathbf{k} \right\vert \right)!}   \\
\times \prod \limits_{j=1}^{d}\left \{  \frac{\left(  2\left(
\left \vert \mathbf{k}^{j+1}\right \vert +\mu+\frac{d-j}{2}\right)  \right)
_{k_{j}}}{k_{j}!} \Theta_{j}^{d}\left(  a,\mu,\mathbf{k};\xi_{j}\right)
\right \} \Lambda \left(  n,\mathbf{k},b,\mu,\beta ,\xi_{d+1}\right)  \label{18b}%
\end{multline}
where
\begin{align*}
\Lambda \left(  n,\mathbf{k},b,\mu,\beta ,\xi_{d+1}\right) = \hyper{2}{1}{-n+\left \vert \mathbf{k}\right \vert ,~b+\left \vert \mathbf{k}\right \vert -i\xi_{d+1}}{2\left \vert \mathbf{k}\right \vert +2\mu+\beta+d}{2},
\end{align*}
and
\begin{multline*}
\Theta_{j}^{d}\left(  a,\mu,\mathbf{k};\xi_{j}\right)  =B\left(
a+\frac{\left \vert \mathbf{k}^{j+1}\right \vert +i\xi_{j}}{2}+\frac{d-j}%
{4},a+\frac{\left \vert \mathbf{k}^{j+1}\right \vert -i\xi_{j}}{2}+\frac{d-j}%
{4}\right)   \\
\times \hyper{3}{2}{-k_{j},\ k_{j}+2\left(  \left \vert \mathbf{k}^{j+1}\right \vert +\mu+\frac{d-j}{2}\right)  ,\ a+\frac{\left \vert \mathbf{k}^{j+1}\right \vert +i\xi_{j}}{2}+\frac{d-j}{4}}{\left \vert \mathbf{k}^{j+1}\right \vert +\mu+\frac{d-j+1}{2},\  \left \vert \mathbf{k} ^{j+1}\right \vert +2a+\frac{d-j}{2}}{1} ,
\end{multline*}
or in terms of the continuous Hahn polynomials (\ref{hahn})%
\begin{align*}
\Theta_{j}^{d}\left(  a,\mu,\mathbf{k};\xi_{j}\right)  =\frac{k_{j}!}%
{i^{k_{j}}\left(  \left \vert \mathbf{k}^{j+1}\right \vert +\mu+\frac{d-j+1}%
{2}\right)  _{k_{j}}\left(  \left \vert \mathbf{k}^{j+1}\right \vert
+2a+\frac{d-j}{2}\right)  _{k_{j}}}  & \\
\times B\left(  a+\frac{\left \vert \mathbf{k}^{j+1}\right \vert +i\xi_{j}}%
{2}+\frac{d-j}{4},a+\frac{\left \vert \mathbf{k}^{j+1}\right \vert -i\xi_{j}}%
{2}+\frac{d-j}{4}\right)   & \\
\times p_{k_{j}}\left(  \frac{\xi_{j}}{2};a+\frac{\left \vert \mathbf{k}%
^{j+1}\right \vert }{2}+\frac{d-j}{4},\mu-a+\frac{\left \vert \mathbf{k}%
^{j+1}\right \vert +1}{2}+\frac{d-j}{4}\right.   & \\
\left.  ,\mu-a+\frac{\left \vert \mathbf{k}^{j+1}\right \vert +1}{2}+\frac
{d-j}{4},a+\frac{\left \vert \mathbf{k}^{j+1}\right \vert }{2}+\frac{d-j}%
{4}\right)  .  &
\end{align*}

\end{theorem}

\begin{proof}
It follows from (\ref{100})
\begin{multline}
\mathcal{F}\left(  g_{n,\mathbf{k}}\left( t, \boldsymbol{x};a,b,\beta,\mu \right) \right)   ={\displaystyle \int \limits_{-\infty}^{\infty}}
{\displaystyle \int \limits_{-\infty}^{\infty}}\cdots {\displaystyle \int \limits_{-\infty}^{\infty}}
g_{n,\mathbf{k}}\left( t,\boldsymbol{x};a,b,\beta,\mu \right) e^{-i\left(  \xi_{1}x_{1}+\cdots+\xi_{d}x_{d}+\xi_{d+1}t\right)  }d\boldsymbol{x}dt \nonumber \\
    ={\displaystyle \int \limits_{0}^{\infty}}
{\displaystyle \int \limits_{-\infty}^{\infty}}\cdots {\displaystyle \int \limits_{-\infty}^{\infty}}
e^{-i\left(  \xi_{1}x_{1}+\cdots+\xi_{d}x_{d}\right)}f_{d}\left(
x_{1},\dots,x_{d};k_{1},\dots,k_{d},a,\mu \right) \nonumber  \\
   \times \exp(-u/2) L_{n-\left \vert \mathbf{k}\right \vert }^{2\left \vert \mathbf{k}\right \vert+2\mu+\beta+d-1}\left(u\right) u^{b+\left \vert \mathbf{k}\right \vert -i\xi_{d+1}-1}d\boldsymbol{x}du \nonumber\\
   =\mathcal{F}\left(f_{d}\left(
x_{1},\dots,x_{d};k_{1},\dots,k_{d},a,\mu \right)\right)\nonumber\\
   \times {\displaystyle \int \limits_{0}^{\infty}}\exp(-u/2) L_{n-\left \vert \mathbf{k}\right \vert }^{2\left \vert \mathbf{k}\right \vert+2\mu+\beta+d-1}\left(u\right) u^{b+\left \vert \mathbf{k}\right \vert -i\xi_{d+1}-1}du\nonumber.
\end{multline}
From the definition of Laguerre polynomials (\ref{102}) and Gamma function, we can write
\begin{multline}
\mathcal{F}\left(  g_{n,\mathbf{k}}\left( t,\boldsymbol{x};a,b,\beta,\mu \right) \right)  = \mathcal{F}\left(f_{d}\left(
x_{1},\dots,x_{d};k_{1},\dots,k_{d},a,\mu \right)\right) \frac{\left(  2\left \vert \mathbf{k}\right \vert +2\mu+\beta+d\right)
_{n-\left \vert \mathbf{k}\right \vert }}{\left(  n-\left \vert \mathbf{k}
\right \vert \right)  !} \nonumber
\\
   \times {\displaystyle \sum \limits_{j=0}^{n-\left \vert \mathbf{k}\right \vert }}
\frac{\left(  -n+\left \vert \mathbf{k}\right \vert \right)  _{j}}{\left(
2\left \vert \mathbf{k}\right \vert +2\mu+\beta+d\right)  _{j}~j!}{\displaystyle \int \limits_{0}^{\infty}}
u^{j+b+\left \vert \mathbf{k}\right \vert -i\xi_{d+1}-1}e^{-u/2}du \nonumber \\
   =\mathcal{F}\left(f_{d}\left(
x_{1},\dots,x_{d};k_{1},\dots,k_{d},a,\mu \right)\right) \frac{\left(  2\left \vert \mathbf{k}\right \vert +2\mu+\beta+d\right)
_{n-\left \vert \mathbf{k}\right \vert }}{\left(  n-\left \vert \mathbf{k}%
\right \vert \right)  !} \nonumber \\
   \times 2^{b+\left \vert \mathbf{k}\right \vert -i\xi_{d+1}%
}\Gamma \left(  b+\left \vert \mathbf{k}\right \vert -i\xi_{d+1}\right)
~_{2}F_{1}\left(\genfrac{.}{\vert}{0pt}{}{-n+\left \vert \mathbf{k}\right \vert ,~b+\left \vert
\mathbf{k}\right \vert -i\xi_{d+1}}{2\left \vert \mathbf{k}\right \vert
+2\mu+\beta+d}~2\right). \nonumber
\end{multline}
The proof follows from \eqref{18}.

Particularly, in the case $d=1$, the Fourier transform of the function
\begin{align}
g_{n,k_{1}}\left( t, x_{1};a,b,\beta,\mu \right)    &  =\left(  1-\tanh^{2}x_{1}\right)^{a} {\mathsf{L}}_{k_{1},n}^{\beta,\mu}\left(\tau_{1}, \tau_{2}\right) \exp(-\frac{e^{t}}{2}+bt) \nonumber \\
   &  =\left(  1-\tanh^{2}x_{1}\right)^{a} L_{n-k_{1}}^{2k_{1}+2\mu+\beta}\left( e^{t} \right) \exp(-\frac{e^{t}}{2}+bt+k_{1}t)
\ C_{k_{1}}^{\left(  \mu \right)  }\left(  \tanh x_{1}\right),\nonumber
\label{g1dim}%
\end{align}
where $\tau_{1}=e^{t}$ and $\tau_{2}=e^{t}\tanh x_{1},$ is given by
\begin{multline}
\mathcal{F}
\left( g_{n,k_{1}}\left( t, x_{1};a,b,\beta,\mu \right) \right)   \\
=\int \limits_{-\infty}^{\infty}\int \limits_{-\infty}^{\infty}e^{-i\xi_{1}x_{1}-i\xi_{2}t}\left(  1-\tanh^{2}x_{1}\right)^{a} \exp(-\frac{e^{t}}{2}+bt+k_{1}t)  L_{n-k_{1}}^{2k_{1}+2\mu+\beta}\left( e^{t} \right)
\ C_{k_{1}}^{\left(  \mu \right)  }\left(  \tanh x_{1}\right) dx_{1}dt \nonumber \\
  =\frac{2^{2a+b+k_{1}-i\xi_{2}-1}\left(  2k_{1}+2\mu+\beta+1\right)  _{n-k_{1}}~\left(
2\mu \right)  _{k_{1}}}{k_{1}!\left(  n-k_{1}\right)  !} \\ \times \Gamma \left(  b+k_{1}-i\xi_{2}\right)
   \Lambda \left(  n,k_{1},b,\mu,\beta ,\xi_{2}\right) \Theta_{1}^{1}\left(  a,\mu,k_{1};\xi_{1}\right)
\end{multline}
where%
\[
\Theta_{1}^{1}\left(  a,\mu,k_{1};\xi_{1}\right)  =B\left(  a+\frac{i\xi_{1}}{2},a-\frac{i\xi_{1}}{2}\right) ~_{3}F_{2}\left(\genfrac{.}{\vert}{0pt}{}{-k_{1},~k_{1}+2\mu,~a+\frac{i\xi_{1}}{2}}{\mu+\frac{1}%
{2},~2a}~1\right),
\]
and
\[
\Lambda \left(  n,k_{1},b,\mu,\beta ,\xi_{2}\right) =_{2}F_{1}\left(\genfrac{.}{\vert}{0pt}{}{-n+k_{1},~b+k_{1}-i\xi_{2}}{2k_{1}+2\mu+\beta+1}~2\right).
\]

It can be rewritten in terms of the continuous Hahn polynomials $p_{n}\left(  x;a,b,c,d\right)$ defined in \eqref{hahn} as
\begin{align*}%
\mathcal{F}\left( g_{n,k_{1}}\left( t, x_{1};a,b,\beta,\mu \right) \right)  &  =\frac
{2^{2a+b+k_{1}-1-i\xi_{2}}\left(  2k_{1}+2\mu+\beta+1\right)  _{n-k_{1}}\left(  2\mu \right)  _{k_{1}}\Gamma \left(  b+k_{1}-i\xi_{2}\right)}{i^{k_{1}}\left(  n-k_{1}\right)  !\left(  2a\right)  _{k_{1}}\left(  \mu+1/2\right)  _{k_{1}}} \\
&  \times B\left(  a+\frac{i\xi_{1}}{2},\ a-\frac
{i\xi_{1}}{2}\right) \Lambda \left(  n,k_{1},b,\mu,\beta ,\xi_{2}\right)\\
&  \times p_{k_{1}}\left(  \frac{\xi_{1}}{2};a,\mu-a+1/2,\mu-a+1/2,a\right) .
\end{align*}

\end{proof}

\subsection{The class of special functions using Fourier transform of the
Laguerre polynomials on the cone}

The Parseval's identity corresponding to (\ref{16}) is given by the statement
\cite[p.118, Eq. (7.17)]{3}
\begin{equation}
\int \limits_{-\infty}^{\infty}g\left(  x\right)  \overline {h\left(  x\right)
}dx=\frac{1}{2\pi}\int \limits_{-\infty}^{\infty}\mathcal{F}
\left(  g\left(  x\right)  \right)  \overline{ \mathcal{F}\left(  h\left(  x\right)  \right)  }d\xi,\label{monoparseval}%
\end{equation}
and in $d-$variable case, the Parseval's identity corresponding to (\ref{17})
is as follows \cite[p. 183, (iv)]{3}
\begin{multline}
 \int \limits_{-\infty}^{\infty}\cdots \int \limits_{-\infty}^{\infty}g\left(
x_{1},\dots,x_{d}\right)  \overline{h\left(  x_{1},\dots,x_{d}\right)  }
dx_{1}\cdots dx_{d} \\
 =\frac{1}{\left(  2\pi \right)  ^{d}}\int \limits_{-\infty}^{\infty}
\cdots \int \limits_{-\infty}^{\infty}\mathcal{F}\left(  g\left(  x_{1},\dots,x_{d}\right)  \right)  \overline{
\mathcal{F}\left(  h\left(  x_{1},\dots,x_{d}\right)  \right)  }d\xi_{1}\cdots d\xi_{d}.\label{multi}
\end{multline}

\begin{theorem}
Let $\mathbf{k}$ and $\mathbf{k}^{j}$be as in \eqref{notation}. Let
$\mathbf{a=}\left(  a_{1},a_{2}\right)  $, $\left \vert \mathbf{a}\right \vert
=a_{1}+a_{2},~\mathbf{b=}\left(  b_{1},b_{2}\right)  $ and $\left \vert
\mathbf{b}\right \vert =b_{1}+b_{2}$. The following equality is satisfied%
\begin{multline*}
  \int \limits_{-\infty}^{\infty}\cdots \int \limits_{-\infty}^{\infty}%
\  \Gamma \left(  b_{1}-it\right)  \Gamma \left(  b_{2}+it\right)
A_{n\boldsymbol{,k}}^{\left(  d+1\right)  }\left( it, i\boldsymbol{x,};a_{1},a_{2},b_{1},b_{2}\right) \\
\times \ A_{m\boldsymbol{,l}}^{\left(  d+1\right)
}\left( -it, -i\boldsymbol{x};a_{2},a_{1},b_{2},b_{1}\right)  d\boldsymbol{x}dt\\
  =\left(  2\pi \right)  ^{d+1}2^{-2d\left \vert \mathbf{a}\right \vert
-2\left \vert \mathbf{k}\right \vert -\left \vert \mathbf{b}\right \vert
+d+1}h_{\mathbf{k}}^{\left(  a_{1}+a_{2}-\frac{1}{2}\right)  }\frac
{\Gamma \left(  \left \vert \mathbf{b}\right \vert +n+\left \vert \mathbf{k}%
\right \vert \right)  \left(  n-\left \vert \mathbf{k}\right \vert \right)
!}{(\left(  2\left \vert \mathbf{k}\right \vert +\left \vert \mathbf{b}\right \vert
\right)  _{n-\left \vert \mathbf{k}\right \vert })^{2}}\delta_{n,m}\\
  \times \prod \limits_{j=1}^{d}\frac{\left(  k_{j}!\right)  ^{2}\Gamma \left(
\left \vert \mathbf{k}^{j+1}\right \vert +2a_{1}+\frac{d-j}{2}\right)
\Gamma \left(  \left \vert \mathbf{k}^{j+1}\right \vert +2a_{2}+\frac{d-j}%
{2}\right)  }{2^{2\left \vert \mathbf{k}^{j+1}\right \vert }\left(  \left(
2\left \vert \mathbf{k}^{j+1}\right \vert +2\left \vert \mathbf{a}\right \vert
+d-j-1\right)  _{k_{j}}\right)  ^{2}}\delta_{k_{j},l_{j}}%
\end{multline*}
for $a_{1},a_{2},b_{1},b_{2}>0$ where $h_{\mathbf{k}}^{\left(  a_{1}%
+a_{2}-\frac{1}{2}\right)  }$ is given by (\ref{Norm}) and%
\begin{align*}
&  \  \ A_{n\boldsymbol{,k}}^{\left(  d+1\right)  }\left( t, \boldsymbol{x}%
;a_{1},a_{2},b_{1},b_{2}\right)  =~_{2}F_{1}\left(
%TCIMACRO{\QDATOP{-n+\left\vert \QTR{bf}{k}\right\vert ,\ b_{1}+\left\vert
%\QTR{bf}{k}\right\vert -t}{2\left\vert \QTR{bf}{k}\right\vert +\left\vert
%\QTR{bf}{b}\right\vert +1}}%
%BeginExpansion
\genfrac{}{}{0pt}{0}{-n+\left \vert \mathbf{k}\right \vert ,\ b_{1}+\left \vert
\mathbf{k}\right \vert -t}{2\left \vert \mathbf{k}\right \vert +\left \vert
\mathbf{b}\right \vert}%
%EndExpansion
\mid2\right)  \left(  b_{1}-t\right)  _{\left \vert \mathbf{k}\right \vert }\\
&  \times \prod \limits_{j=1}^{d}\left \{  \Gamma \left(  a_{1}+\frac{\left \vert
\mathbf{k}^{j+1}\right \vert -x_{j}}{2}+\frac{d-j}{4}\right)  \Gamma \left(
a_{1}+\frac{\left \vert \mathbf{k}^{j+1}\right \vert +x_{j}}{2}+\frac{d-j}%
{4}\right)  \right. \\
&  \times \left.  _{3}F_{2}\left(
%TCIMACRO{\QDATOP{-k_{j},\ k_{j}+2\left(  \left\vert \QTR{bf}{k}^{j+1}%
%\right\vert +\left\vert \QTR{bf}{a}\right\vert +\frac{d-j-1}{2}\right)
%,\ a_{1}+\frac{\left\vert \QTR{bf}{k}^{j+1}\right\vert +x_{j}}{2}+\frac
%{d-j}{4}}{\left\vert \QTR{bf}{k}^{j+1}\right\vert +\left\vert
%\QTR{bf}{a}\right\vert +\frac{d-j}{2},\ \left\vert \QTR{bf}{k}^{j+1}%
%\right\vert +2a_{1}+\frac{d-j}{2}}}%
%BeginExpansion
\genfrac{}{}{0pt}{0}{-k_{j},\ k_{j}+2\left(  \left \vert \mathbf{k}%
^{j+1}\right \vert +\left \vert \mathbf{a}\right \vert +\frac{d-j-1}{2}\right)
,\ a_{1}+\frac{\left \vert \mathbf{k}^{j+1}\right \vert +x_{j}}{2}+\frac{d-j}%
{4}}{\left \vert \mathbf{k}^{j+1}\right \vert +\left \vert \mathbf{a}\right \vert
+\frac{d-j}{2},\  \left \vert \mathbf{k}^{j+1}\right \vert +2a_{1}+\frac{d-j}{2}}%
%EndExpansion
\mid1\right)  \right \}
\end{align*}
or, in terms of Hahn polynomials (\ref{hahn})%
\[%
\begin{array}
[b]{l}%
\ A_{n\boldsymbol{,k}}^{\left(  d+1\right)  }\left( t, \boldsymbol{x};a_{1},a_{2},b_{1},b_{2}\right)  =~_{2}F_{1}\left(
%TCIMACRO{\QDATOP{-n+\left\vert \QTR{bf}{k}\right\vert ,\ b_{1}+\left\vert
%\QTR{bf}{k}\right\vert -t}{2\left\vert \QTR{bf}{k}\right\vert +\left\vert
%\QTR{bf}{b}\right\vert +1}}%
%BeginExpansion
\genfrac{}{}{0pt}{0}{-n+\left \vert \mathbf{k}\right \vert ,\ b_{1}+\left \vert
\mathbf{k}\right \vert -t}{2\left \vert \mathbf{k}\right \vert +\left \vert
\mathbf{b}\right \vert}%
%EndExpansion
\mid2\right)  \left(  b_{1}-t\right)  _{\left \vert \mathbf{k}\right \vert }\\
\times%
%TCIMACRO{\dprod \limits_{j=1}^{d}}%
%BeginExpansion
{\displaystyle \prod \limits_{j=1}^{d}}
%EndExpansion
\left \{  \dfrac{k_{j}!i^{-k_{j}}}{\left(  \left \vert \mathbf{k}^{j+1}%
\right \vert +2a_{1}+\dfrac{d-j}{2}\right)  _{k_{j}}\left(  \left \vert
\mathbf{k}^{j+1}\right \vert +\left \vert \mathbf{a}\right \vert +\dfrac{d-j}%
{2}\right)  _{k_{j}}}\right. \\
\times p_{k_{j}}\left(  -\dfrac{ix_{j}}{2};a_{1}+\dfrac{\left \vert
\mathbf{k}^{j+1}\right \vert }{2}+\dfrac{d-j}{4},a_{2}+\dfrac{\left \vert
\mathbf{k}^{j+1}\right \vert }{2}+\dfrac{d-j}{4}\right. \\
\left.  \left.  ,a_{2}+\dfrac{\left \vert \mathbf{k}^{j+1}\right \vert }%
{2}+\dfrac{d-j}{4},a_{1}+\dfrac{\left \vert \mathbf{k}^{j+1}\right \vert }%
{2}+\dfrac{d-j}{4}\right)  \right. \\
\times \left.  \Gamma \left(  a_{1}+\dfrac{\left \vert \mathbf{k}^{j+1}%
\right \vert -x_{j}}{2}+\dfrac{d-j}{4}\right)  \Gamma \left(  a_{1}%
+\dfrac{\left \vert \mathbf{k}^{j+1}\right \vert +x_{j}}{2}+\dfrac{d-j}%
{4}\right)  \right \}
\end{array}
\]
for $d\geq1$.
\end{theorem}

\begin{proof}
The proof follows by using the induction method. Let's start with the case $d=1$. In
this case, we get the specific function from (\ref{100})%
\begin{multline*}
g_{n,k_{1}}\left( t, x_{1};a_{1},b_{1},\beta_{1},\mu_{1}\right)    =\left(
1-\tanh^{2}x_{1}\right)  ^{a_{1}}L_{k_{1},n}^{\beta_{1},\mu_{1}}\left(
\tau_{1},\tau_{2}\right)  \exp \left(  -e^{t}/2+b_{1}t\right) \\
  =\left(  1-\tanh^{2}x_{1}\right)  ^{a_{1}}\ L_{n-k_{1}}^{2k_{1}+2\mu
_{1}+\beta_{1}}\left(  e^{t}\right)  \exp \left(  -e^{t}/2+b_{1}t+k_{1}%
t\right)  C_{k_{1}}^{\left(  \mu_{1}\right)  }\left(  \tanh x_{1}\right)
\end{multline*}
where $\tau_{1}=e^{t}, \tau_{2}=e^{t}\tanh x_{1}.$ By Parseval's identity%
\begin{multline*}
  4\pi^{2}%
%TCIMACRO{\dint \limits_{-\infty}^{\infty}}%
%BeginExpansion
{\displaystyle \int \limits_{-\infty}^{\infty}}
%EndExpansion%
%TCIMACRO{\dint \limits_{-\infty}^{\infty}}%
%BeginExpansion
{\displaystyle \int \limits_{-\infty}^{\infty}}
%EndExpansion
g_{n,k_{1}}\left( t, x_{1};a_{1},b_{1},\beta_{1},\mu_{1}\right)  g_{m,l_{1}%
}\left( t, x_{1};a_{2},b_{2},\beta_{2},\mu_{2}\right)  dx_{1}dt\\
  =\int \limits_{-\infty}^{\infty}\int \limits_{-\infty}^{\infty}%
%TCIMACRO{\tciFourier}%
%BeginExpansion
\mathcal{F}%
%EndExpansion
\left(  g_{n,k_{1}}\left( t, x_{1};a_{1},b_{1},\beta_{1},\mu_{1}\right)
\right)  \overline{%
%TCIMACRO{\tciFourier}%
%BeginExpansion
\mathcal{F}%
%EndExpansion
\left(  g_{m,l_{1}}\left( t, x_{1};a_{2},b_{2},\beta_{2},\mu_{2}\right)
\right)  }d\xi_{1}d\xi_{2},
\end{multline*}
it follows%
\begin{multline*}
4\pi^{2}%
{\displaystyle \int \limits_{-\infty}^{\infty}}
{\displaystyle \int \limits_{-\infty}^{\infty}}
%EndExpansion
\left(  1-\tanh^{2}x_{1}\right)  ^{a_{1}+a_{2}}\exp \left(  -e^{t}+\left(
b_{1}+b_{2}\right)  t\right)  L_{k_{1},n}^{\beta_{1},\mu_{1}}\left(  \tau
_{1}, \tau_{2}\right)  L_{l_{1},m}^{\beta_{2},\mu_{2}}\left(  \tau_{1}%
, \tau_{2}\right)  dx_{1}dt\\
  =4\pi^{2}%
{\displaystyle \int \limits_{-\infty}^{\infty}}
{\displaystyle \int \limits_{-\infty}^{\infty}}
%EndExpansion
\left(  1-\tanh^{2}x_{1}\right)  ^{a_{1}+a_{2}}\exp \left(  -e^{t}+\left(
b_{1}+b_{2}\right)  t+\left(  k_{1}+l_{1}\right)  t\right) \\
  \times L_{n-k_{1}}^{2k_{1}+2\mu_{1}+\beta_{1}}\left(  e^{t}\right)
L_{m-l_{1}}^{2l_{1}+2\mu_{2}+\beta_{2}}\left(  e^{t}\right)  C_{k_{1}%
}^{\left(  \mu_{1}\right)  }\left(  \tanh x_{1}\right)  C_{l_{1}}^{\left(
\mu_{2}\right)  }\left(  \tanh x_{1}\right)  dx_{1}dt\\
  =4\pi^{2}%
{\displaystyle \int \limits_{0}^{\infty}}
%EndExpansion
L_{n-k_{1}}^{2k_{1}+2\mu_{1}+\beta_{1}}\left(  u\right)  L_{m-l_{1}}%
^{2l_{1}+2\mu_{2}+\beta_{2}}\left(  u\right)  e^{-u}u^{b_{1}+b_{2}+k_{1}%
+l_{1}-1}du\\
  \times%
{\displaystyle \int \limits_{-1}^{1}}
%EndExpansion
\left(  1-v^{2}\right)  ^{a_{1}+a_{2}-1}C_{k_{1}}^{\left(  \mu_{1}\right)
}\left(  v\right)  C_{l_{1}}^{\left(  \mu_{2}\right)  }\left(  v\right)  dv\\
  =\frac{2^{2\left(  a_{1}+a_{2}-1\right)  +b_{1}+b_{2}+k_{1}+l_{1}}\left(
2\mu_{1}\right)  _{k_{1}}\left(  2\mu_{2}\right)  _{l_{1}}\left(  2k_{1}%
+2\mu_{1}+\beta_{1}+1\right)  _{n-k_{1}}\left(  2l_{1}+2\mu_{2}+\beta
_{2}+1\right)  _{m-l_{1}}}{\left(  n-k_{1}\right)  !\left(  m-l_{1}\right)
!k_{1}!l_{1}!\Gamma \left(  2a_{1}\right)  \Gamma \left(  2a_{2}\right)  }\\
  \times%
{\displaystyle \int \limits_{-\infty}^{\infty}}
{\displaystyle \int \limits_{-\infty}^{\infty}}
%EndExpansion
\text{ }\Gamma \left(  a_{1}+\frac{i\xi_{1}}{2}\right)  \Gamma \left(
a_{1}-\frac{i\xi_{1}}{2}\right)  ~\overline{\Gamma \left(  a_{2}+\frac{i\xi
_{1}}{2}\right)  \Gamma \left(  a_{2}-\frac{i\xi_{1}}{2}\right)  }\\
  \times \Gamma \left(  b_{1}-i\xi_{2}\right)  ~\overline{\Gamma \left(
b_{2}-i\xi_{2}\right)  }\left(  b_{1}-i\xi_{2}\right)  _{k_{1}}\overline
{\left(  b_{2}-i\xi_{2}\right)  _{l_{1}}}\\
  \times~_{3}F_{2}\left(
\genfrac{}{}{0pt}{0}{-k_{1},\ k_{1}+2\mu_{1},\ a_{1}+\frac{i\xi_{1}}%
{2}}{2a_{1},\  \mu_{1}+1/2}%
%EndExpansion
\mid1\right)  ~\overline{_{3}F_{2}\left(
\genfrac{}{}{0pt}{0}{-l_{1},\ l_{1}+2\mu_{2},\ a_{2}+\frac{i\xi_{1}}%
{2}}{2a_{2},\  \mu_{2}+1/2}%
%EndExpansion
\mid1\right)  }\\
  \times~_{2}F_{1}\left(
\genfrac{}{}{0pt}{0}{-n+k_{1},\ b_{1}+k_{1}-i\xi_{2}}{2k_{1}+2\mu_{1}%
+\beta_{1}+1}%
%EndExpansion
\mid2\right)  ~\overline{_{2}F_{1}\left(
\genfrac{}{}{0pt}{0}{-m+l_{1},\ b_{2}+l_{1}-i\xi_{2}}{2l_{1}+2\mu_{2}%
+\beta_{2}+1}%
%EndExpansion
\mid2\right)  }d\xi_{1}d\xi_{2}.
\end{multline*}
By assuming
\begin{align*}
\mu_{1}  &  =\mu_{2}=a_{1}+a_{2}-\frac{1}{2}\\
\beta_{1}  &  =\beta_{2}=b_{1}+b_{2}-2a_{1}-2a_{2}
\end{align*}
and considering the orthogonality relations of (\ref{ort}) and (\ref{ort-L}),
it is seen that the special function%
\begin{multline*}
\ A_{n,k_{1}}^{\left(  2\right)  }\left( t, x_{1};a_{1},a_{2},b_{1}%
,b_{2}\right)   = \, \hyper{3}{2}{-k_{1},\ k_{1}+2\left(  a_{1}+a_{2}\right)-1,\ a_{1}+\frac{x_{1}}{2}}{a_{1}+a_{2},\ 2a_{1}}{1} \\
  \times \Gamma \left(  a_{1}-\frac{x_{1}}{2}\right)
\Gamma \left(  a_{1}+\frac{x_{1}}{2}\right)  \left(  b_{1}-t\right)  _{k_{1}} \hyper{2}{1}{-n+k_{1},\ b_{1}+k_{1}-t}{2k_{1}+b_{1}+b_{2}}{2} \\
  =\frac{k_{1}!i^{-n_{1}}}{\left(  2a_{1}\right)  _{k_{1}}\left(  a_{1}+a_{2}\right)  _{k_{1}}}\Gamma \left(  a_{1}-\frac{x_{1}}{2}\right) \Gamma \left(  a_{1}+\frac{x_{1}}{2}\right)  \left(  b_{1}-t\right)  _{k_{1}}\\
  \times~\hyper{2}{1}{-n+k_{1},\ b_{1}+k_{1}-t}{2k_{1}+b_{1}+b_{2}}{2} p_{k_{1}}\left(  \frac{-ix_{1}}{2};a_{1},a_{2},a_{2},a_{1}\right)
\end{multline*}
has the orthogonality relation%
\begin{multline*}
\int \limits_{-\infty}^{\infty}\int \limits_{-\infty}^{\infty}\Gamma \left(
b_{1}-it\right)  \Gamma \left(  b_{2}+it\right)  \ A_{n,k_{1}}^{\left(
2\right)  }\left( it, ix_{1};a_{1},a_{2},b_{1},b_{2}\right)  \\ \times A_{m,l_{1}%
}^{\left(  2\right)  }\left(  -it,-ix_{1};a_{2},a_{1},b_{2},b_{1}\right)
dx_{1}dt \\
=\frac{4\pi^{2}2^{-\left(  2a_{1}+2a_{2}+b_{1}+b_{2}+2k_{1}\right)
+2}h_{k_{1}}^{\left(  a_{1}+a_{2}-\frac{1}{2}\right)  }\left(  k_{1}!\right)
^{2}\left(  n-k_{1}\right)  !}{\left(  2k_{1}+b_{1}+b_{2}\right)  _{n-k_{1}%
}^{2}\left(  2a_{1}+2a_{2}-1\right)  _{k_{1}}^{2}} \\ \times \Gamma \left(  b_{1}+b_{2}+n+k_{1}\right)  \Gamma \left(
2a_{1}\right)  \Gamma \left(  2a_{2}\right)  \delta_{k_{1},l_{1}}\delta_{n,m}%
\end{multline*}
where $h_{k_{1}}^{\left(  a_{1}+a_{2}-\frac{1}{2}\right)  }$ is given by
(\ref{gnorm}). Similarly, the result follows by iteration if we substitute (\ref{15}) and (\ref{18}) in
Parseval's identity (\ref{multi}).
\end{proof}

\subsection{The Fourier transform of Jacobi polynomials on the cone}
Let us consider a function in terms of Jacobi polynomials on the cone \eqref{eq:jacobicone} as

\begin{multline}
g_{n,\mathbf{k}}\left( t, \boldsymbol{x};a,b,c,\beta,\mu,\gamma \right)    :=g_{n,k_{1},\dots,k_{d}}\left(
t,x_{1},\dots,x_{d};a,b,c,\beta,\mu,\gamma \right)  \\
 =\prod \limits_{j=1}^{d}\left(  1-\tanh^{2}x_{j}\right)  ^{a+\frac{d-j}{4}%
}(1+\tanh t)^{b}(1-\tanh t)^{c}{\mathsf{J}}_{{\mathbf{k}},n}^{\beta,\mu,\gamma}\left(  \tau_{1},\dots,\tau_{d},\tau_{d+1}\right),
\end{multline}
for $d\geq1$, where $a,b,c,\beta,\mu,\gamma$ real parameters and
\begin{align*}
\tau_{1}\left( t \right) & =\tau_{1}=\left( \frac{1+\tanh t}{2}\right) \\
\tau_{2}\left( t, x_{1}\right)   & =\tau_{2}=\left( \frac{1+\tanh t}{2}\right) \tanh x_{1},\\
\tau_{d+1}\left( t, x_{1},\dots,x_{d}\right)   & =\tau_{d+1}=\left( \frac{1+\tanh t}{2}\right) \\& \times \tanh x_{d}%
\sqrt{\left(  1-\tanh^{2}x_{1}\right)  \left(  1-\tanh^{2}x_{2}\right)
\dots \left(  1-\tanh^{2}x_{d-1}\right)  },
\end{align*}
for $d\geq1$.
In explicit form, we can write
\begin{multline*}
g_{n,\mathbf{k}}\left( t, \boldsymbol{x};a,b,c,\beta,\mu,\gamma \right)  =\frac{1}{2^{\left \vert \mathbf{k}\right \vert} }(1+\tanh t)^{b+\left \vert \mathbf{k}\right \vert}(1-\tanh t)^{c}\ P_{n-\left \vert \mathbf{k}\right \vert }^{(2\left \vert \mathbf{k}\right \vert+2\mu+\beta+d-1,\gamma)}\left(-\tanh t\right)\\
\times \prod \limits_{j=1}^{d}\left(  1-\tanh^{2}x_{j}\right)  ^{a+\frac{d-j}{4}} \prod \limits_{j=1}^{d-1}\left(  1-\tanh^{2}x_{j}\right)  ^\frac{k_{j+1}+\cdots+k_{d}}{2}
\prod \limits_{j=1}^{d}C_{k_{j}}^{\left(  \lambda_{j} \right)}\left(\tanh x_{j}\right)
\end{multline*}
where $\lambda_{j}=\mu+\left \vert \mathbf{k}^{j+1}\right \vert +\frac{d-j}{2}$, or
\begin{multline}
g_{n,\mathbf{k}}\left( t, \boldsymbol{x};a,b,c,\beta,\mu,\gamma \right) =\frac{1}{2^{\left \vert \mathbf{k}\right \vert} }(1+\tanh t)^{b+\left \vert \mathbf{k}\right \vert}(1-\tanh t)^{c}\\ \times P_{n-\left \vert \mathbf{k}\right \vert }^{(2\left \vert \mathbf{k}\right \vert+2\mu+\beta+d-1,\gamma)}\left(-\tanh t\right)
 f_{d}\left(
x_{1},\dots,x_{d};k_{1},\dots,k_{d},a,\mu \right)\label{jac-spec-func}
\end{multline}
where $f_{d}\left(x_{1},\dots,x_{d};k_{1},\dots,k_{d},a,\mu \right)$ is defined  in \eqref{150}.

We now calculate the Fourier transform of the function $g_{n,\mathbf{k}}\left( t, \boldsymbol{x};a,b,c,\beta,\mu,\gamma \right)$ given by (\ref{jac-spec-func}) by using Lemma \ref{prop:OPcone2}.

\begin{theorem}
The Fourier transform of the function $g_{n,\mathbf{k}}\left( t, \boldsymbol{x};a,b,c,\beta,\mu,\gamma \right)$ is given explicitly as follows%
\begin{multline}%
\mathcal{F}
\left(  g_{n,\mathbf{k}}\left( t,\boldsymbol{x};a,b,c,\beta,\mu,\gamma \right)  \right)  :=%
\mathcal{F}
\left(  g_{n,\mathbf{k}}\left( t, x_{1},\dots,x_{d};a,b,c,\beta,\mu,\gamma\right)  \right) \\
=2^{b+c-1+2ad+\frac{d\left(  d-5\right)  }{4}+\text{$
{\textstyle \sum \limits_{j=1}^{d-1}}
$}{ jk}_{j+1}}\frac{ \left(  2 \left\vert \mathbf{k} \right\vert + 2\mu + \beta + d \right)_{n - \left\vert \mathbf{k} \right\vert} \Gamma \left( b + \left\vert \mathbf{k} \right\vert - \frac{i\xi_{d+1}}{2} \right)\Gamma \left( c + \frac{i\xi_{d+1}}{2} \right)}{\left( n - \left\vert \mathbf{k} \right\vert \right)!\Gamma \left( \left\vert \mathbf{k} \right\vert +b+c\right)}   \\
\times \prod \limits_{j=1}^{d}\left \{  \frac{\left(  2\left(
\left \vert \mathbf{k}^{j+1}\right \vert +\mu+\frac{d-j}{2}\right)  \right)
_{k_{j}}}{k_{j}!} \Theta_{j}^{d}\left(  a,\mu,\mathbf{k};\xi_{j}\right)
\right \} \Xi \left(  n,\mathbf{k},b,c,\mu,\beta,\gamma ,\xi_{d+1}\right)  \label{Jac-Fourier}%
\end{multline}
where
\begin{align*}
\Xi \left(  n,\mathbf{k},b,c,\mu,\beta,\gamma ,\xi_{d+1}\right) = _{3}F_{2}\left(\genfrac{.}{\vert}{0pt}{}{-n+\left \vert \mathbf{k}\right \vert , n+\left \vert \mathbf{k}\right \vert+2\mu+\beta+\gamma+d, \left \vert \mathbf{k}\right \vert+b -\frac{i\xi_{d+1}}{2}}{2\left \vert \mathbf{k}\right \vert
+2\mu+\beta+d, \left \vert \mathbf{k}\right \vert+b+c}1\right)
\end{align*}
and
\begin{align*}
\Theta_{j}^{d}\left(  a,\mu,\mathbf{k};\xi_{j}\right)  =B\left(
a+\frac{\left \vert \mathbf{k}^{j+1}\right \vert +i\xi_{j}}{2}+\frac{d-j}%
{4},a+\frac{\left \vert \mathbf{k}^{j+1}\right \vert -i\xi_{j}}{2}+\frac{d-j}%
{4}\right)   & \\
\times \ _{3}F_{2}\left(\genfrac{}{}{0pt}{0}{-k_{j},\ k_{j}+2\left(  \left \vert \mathbf{k}%
^{j+1}\right \vert +\mu+\frac{d-j}{2}\right)  ,\ a+\frac{\left \vert
\mathbf{k}^{j+1}\right \vert +i\xi_{j}}{2}+\frac{d-j}{4}}{\left \vert
\mathbf{k}^{j+1}\right \vert +\mu+\frac{d-j+1}{2},\  \left \vert \mathbf{k}%
^{j+1}\right \vert +2a+\frac{d-j}{2}}
\mid1\right)  ,  &
\end{align*}
or in terms of the continuous Hahn polynomials (\ref{hahn})%
\begin{align*}
\Theta_{j}^{d}\left(  a,\mu,\mathbf{k};\xi_{j}\right)  =\frac{k_{j}!}%
{i^{k_{j}}\left(  \left \vert \mathbf{k}^{j+1}\right \vert +\mu+\frac{d-j+1}%
{2}\right)  _{k_{j}}\left(  \left \vert \mathbf{k}^{j+1}\right \vert
+2a+\frac{d-j}{2}\right)  _{k_{j}}}  & \\
\times B\left(  a+\frac{\left \vert \mathbf{k}^{j+1}\right \vert +i\xi_{j}}%
{2}+\frac{d-j}{4},a+\frac{\left \vert \mathbf{k}^{j+1}\right \vert -i\xi_{j}}%
{2}+\frac{d-j}{4}\right)   & \\
\times p_{k_{j}}\left(  \frac{\xi_{j}}{2};a+\frac{\left \vert \mathbf{k}%
^{j+1}\right \vert }{2}+\frac{d-j}{4},\mu-a+\frac{\left \vert \mathbf{k}%
^{j+1}\right \vert +1}{2}+\frac{d-j}{4}\right.   & \\
\left.  ,\mu-a+\frac{\left \vert \mathbf{k}^{j+1}\right \vert +1}{2}+\frac
{d-j}{4},a+\frac{\left \vert \mathbf{k}^{j+1}\right \vert }{2}+\frac{d-j}%
{4}\right)  .  &
\end{align*}

\end{theorem}

\begin{proof}
It follows from (\ref{jac-spec-func})
\begin{multline*}
\mathcal{F}\left(  g_{n,\mathbf{k}}\left( t, \boldsymbol{x};a,b,c,\beta,\mu,\gamma \right) \right)   \\
={\displaystyle \int \limits_{-\infty}^{\infty}}
{\displaystyle \int \limits_{-\infty}^{\infty}}\cdots {\displaystyle \int \limits_{-\infty}^{\infty}}
g_{n,\mathbf{k}}\left( t,\boldsymbol{x};a,b,c,\beta,\mu,\gamma \right) e^{-i\left(  \xi_{1}x_{1}+\cdots +\xi_{d}x_{d}+\xi_{d+1}t\right)  }d\boldsymbol{x}dt  \\
    ={\displaystyle \int \limits_{0}^{1}}
{\displaystyle \int \limits_{-\infty}^{\infty}}\cdots {\displaystyle \int \limits_{-\infty}^{\infty}}
e^{-i\left(  \xi_{1}x_{1}+\cdots +\xi_{d}x_{d}\right)}f_{d}\left(
x_{1},\dots,x_{d};k_{1},\dots,k_{d},a,\mu \right)   \\
   \times 2^{b+c-1} P_{n-\left \vert \mathbf{k}\right \vert }^{(2\left \vert \mathbf{k}\right \vert+2\mu+\beta+d-1,\gamma)}\left(1-2u\right) u^{b+\left \vert \mathbf{k}\right \vert-\frac{i\xi_{d+1}}{2}-1}(1-u)^{c+\frac{i\xi_{d+1}}{2}-1}d\boldsymbol{x}du \\
   =\mathcal{F}\left(f_{d}\left(
x_{1},\dots,x_{d};k_{1},\dots,k_{d},a,\mu \right)\right)\nonumber\\
  \times 2^{b+c-1}{\displaystyle \int \limits_{0}^{1}} P_{n-\left \vert \mathbf{k}\right \vert }^{(2\left \vert \mathbf{k}\right \vert+2\mu+\beta+d-1,\gamma)}\left(1-2u\right) u^{b+\left \vert \mathbf{k}\right \vert-\frac{i\xi_{d+1}}{2}-1}(1-u)^{c+\frac{i\xi_{d+1}}{2}-1}du.
\end{multline*}
From the definition of Jacobi polynomials (\ref{200}) and Gamma function, we can write
\begin{multline}
\mathcal{F} \left(  g_{n,\mathbf{k}}\left( t, \boldsymbol{x};a,b,c,\beta,\mu,\gamma \right) \right)  = \mathcal{F}\left(f_{d}\left(
x_{1},\dots,x_{d};k_{1},\dots ,k_{d},a,\mu \right)\right) \\
\times \frac{2^{b+c-1}\left(  2\left \vert \mathbf{k}\right \vert +2\mu+\beta+d\right)_{n-\left \vert \mathbf{k}\right \vert }}{\left(  n-\left \vert \mathbf{k}
\right \vert \right)  !}
  {\displaystyle \sum \limits_{j=0}^{n-\left \vert \mathbf{k}\right \vert }}
\frac{\left(  -n+\left \vert \mathbf{k}\right \vert \right)  _{j}\left(  n+\left \vert \mathbf{k}\right \vert +2\mu+\beta+\gamma+d\right)
_{j}}{\left(
2\left \vert \mathbf{k}\right \vert +2\mu+\beta+d\right)  _{j}~j!}\\
   {\displaystyle \times \int \limits_{0}^{1}}
u^{j+b+\left \vert \mathbf{k}\right \vert -1-\frac{i\xi_{d+1}}{2}}(1-u)^{c-1+\frac{i\xi_{d+1}}{2}}du \nonumber \\
   =\mathcal{F}\left(f_{d}\left(
x_{1},\dots,x_{d};k_{1},\dots,k_{d},a,\mu \right)\right)2^{b+c-1}\left(  2\left \vert \mathbf{k}\right \vert +2\mu+\beta+d\right)
_{n-\left \vert \mathbf{k}\right \vert } \\
   \times \frac{\Gamma\left(  \left \vert \mathbf{k}\right \vert +b-\frac{i\xi_{d+1}}{2} \right)\Gamma\left( c+\frac{i\xi_{d+1}}{2} \right)
}{\left(  n-\left \vert \mathbf{k}%
\right \vert \right)  !\Gamma\left(  \left \vert \mathbf{k}\right \vert +b+c \right)}  \\
   \times  _{3}F_{2}\left(\genfrac{.}{\vert}{0pt}{}{-n+\left \vert \mathbf{k}\right \vert , n+\left \vert \mathbf{k}\right \vert+2\mu+\beta+\gamma+d, \left \vert \mathbf{k}\right \vert+b -\frac{i\xi_{d+1}}{2}}{2\left \vert \mathbf{k}\right \vert
+2\mu+\beta+d, \left \vert \mathbf{k}\right \vert+b+c}1\right).
\end{multline}
By using the equation (\ref{18}), the proof is completed.

Particularly, in the case $d=1$, the special function is in the form
\begin{align}
g_{n,k_{1}}\left( t, x_{1};a,b,c,\beta,\mu,\gamma  \right)    &  =\left(  1-\tanh^{2}x_{1}\right)^{a}(1+\tanh t)^{b}(1-\tanh t)^{c} {\mathsf{J}}_{k_{1},n}^{\beta,\mu,\gamma}\left(\tau_{1}, \tau_{2}\right)\nonumber \\
   &  =\left(  1-\tanh^{2}x_{1}\right)^{a}(1+\tanh t)^{b}(1-\tanh t)^{c}(\frac{1+\tanh t}{2})^{k_{1}}\\
      &   \times P_{n-k_{1}}^{(2k_{1}+2\mu+\beta,\gamma)}\left( -\tanh t \right)
\ C_{k_{1}}^{\left(  \mu \right)  }\left(  \tanh x_{1}\right),\nonumber
\label{1-dim}%
\end{align}
where $\tau_{1}=\left( \frac{1+\tanh t}{2}\right) $ and $\tau_{2}=\left( \frac{1+\tanh t}{2}\right) \tanh x_{1}.$ The Fourier transform of this function is
\begin{multline}
\mathcal{F}
\left( g_{n,k_{1}}\left( t, x_{1};a,b,c,\beta,\mu,\gamma  \right) \right)    =\int
\limits_{-\infty}^{\infty}\int
\limits_{-\infty}^{\infty}e^{-i\xi_{1}x_{1}-i\xi_{2}t}\left(  1-\tanh^{2}x_{1}\right)^{a}(1+\tanh t)^{b} \\
   \times (1-\tanh t)^{c} \left(\frac{1+\tanh t}{2} \right)^{k_{1}}P_{n-k_{1}}^{(2k_{1}+2\mu+\beta,\gamma)}\left( -\tanh t \right)
\ C_{k_{1}}^{\left(  \mu \right)  }\left(  \tanh x_{1}\right)dx_{1}dt \\
   =2^{b+c+2a-2}\frac{ \left(  2 k_{1} + 2\mu + \beta + 1 \right)_{n-k_{1}} \Gamma \left( b + k_{1} - \frac{i\xi_{2}}{2} \right)\Gamma \left( c + \frac{i\xi_{2}}{2} \right)(2\mu)_{k_{1}}}{\left( n-k_{1} \right)!{k_{1}}!\Gamma \left( k_{1} +b+c\right)}   \\
  \times \Theta_{1}^{1}\left(  a,\mu,k_{1};\xi_{1}\right) \Xi \left(  n,k_{1},b,c,\mu,\beta,\gamma ,\xi_{2}\right)
\end{multline}
where%
\[
\Theta_{1}^{1}\left(  a,\mu,k_{1};\xi_{1}\right)  =B\left(  a+\frac{i\xi_{1}}{2},a-\frac{i\xi_{1}}{2}\right) ~_{3}F_{2}\left(\genfrac{.}{\vert}{0pt}{}{-k_{1},~k_{1}+2\mu,~a+\frac{i\xi_{1}}{2}}{\mu+\frac{1}%
{2},~2a}~1\right),
\]
and
\[
\Xi \left(  n,k_{1},b,c,\mu,\beta,\gamma ,\xi_{2}\right) = _{3}F_{2}\left(\genfrac{.}{\vert}{0pt}{}{-n+k_{1}, n+k_{1}+2\mu+\beta+\gamma+1, k_{1}+b -\frac{i\xi_{2}}{2}}{2k_{1}+2\mu+\beta+1, k_{1}+b+c}1\right)
\]

It can be rewritten in terms of the continuous Hahn polynomials
$p_{n}\left(  x;a,b,c,d\right)  $ as
\begin{multline*}%
\mathcal{F}\left( g_{n,k_{1}}\left( t, x_{1};a,b,c,\beta,\mu,\gamma \right) \right)  \\
= 2^{b+c+2a-2} \frac{ \left( 2 k_{1} + 2\mu + \beta + 1 \right)_{n-k_{1}} \Gamma \left( b + k_{1} - \frac{i\xi_{2}}{2} \right) \Gamma \left( c + \frac{i\xi_{2}}{2} \right) (2\mu)_{k_{1}}}{\Gamma \left( k_{1} + b + c \right) i^{n} (2a)_{k_{1}} \left( \mu + \frac{1}{2} \right)_{k_{1}}}\\
\times \frac{B \left( a + \frac{i\xi_{1}}{2}, a - \frac{i\xi_{1}}{2} \right)}{\left( 2k_{1} + 2\mu + \beta + 1 \right)_{n-k_{1}} \left( k_{1} + b + c \right)_{n-k_{1}} }\\
 \times p_{n-k_{1}}\left( \frac{-\xi_{2}}{2}; k_{1} + b, \gamma - c + 1, k_{1} + 2\mu + \beta + 1 - b, c \right)\\
 \times p_{k_{1}}\left( \frac{\xi_{1}}{2}; a, \mu - a + \frac{1}{2}, \mu - a + \frac{1}{2}, a \right).
\end{multline*}

\end{proof}

\subsection{The class of special functions using Fourier transform of the
Jacobi polynomials on the cone}

We now apply Parseval's identity to the using Fourier transform of the Jacobi polynomials on the cone.

\begin{theorem}
Let $\mathbf{k}$ and $\mathbf{k}^{j}$be as in (\ref{notation}) and, let
$\mathbf{a=}\left(  a_{1},a_{2}\right)  $, $\left \vert \mathbf{a}\right \vert
=a_{1}+a_{2},~\mathbf{b=}\left(  b_{1},b_{2}\right)  $, $\left \vert
\mathbf{b}\right \vert =b_{1}+b_{2}$, $\mathbf{c=}\left(  c_{1},c_{2}\right)  $
and $\left \vert \mathbf{c}\right \vert =c_{1}+c_{2}.$ The following equality is
satisfied%
\begin{multline*}
  \int \limits_{-\infty}^{\infty}\cdots \int \limits_{-\infty}^{\infty}%
\  \Gamma \left(  b_{1}-\frac{it}{2}\right)  \Gamma \left(  b_{2}+\frac{it}%
{2}\right)  \Gamma \left(  c_{1}+\frac{it}{2}\right)  \Gamma \left(  c_{2}%
-\frac{it}{2}\right) \\ \times B_{n,\mathbf{k}}^{\left(  d+1\right)
}\left(  it,i\boldsymbol{x};a_{1},a_{2},b_{1},b_{2},c_{1},c_{2}\right)
   \ B_{m,\mathbf{l}}^{\left(  d+1\right)  }\left(
\boldsymbol{-}it,-i\boldsymbol{x};a_{2},a_{1},b_{2},b_{1},c_{2},c_{1}\right)
d\boldsymbol{x}dt\\
  =\left(  2\pi \right)  ^{d+1}2^{-2d\left \vert \mathbf{a}\right \vert
+d+2}h_{\mathbf{k}}^{\left(  a_{1}+a_{2}-\frac{1}{2}\right)  }\left(
n-\left \vert \mathbf{k}\right \vert \right)  !\\
  \times \frac{\Gamma \left(  n+\left \vert \mathbf{k}\right \vert +\left \vert
\mathbf{b}\right \vert \right)  \Gamma \left(  n-\left \vert \mathbf{k}%
\right \vert +\left \vert \mathbf{c}\right \vert \right)  \Gamma \left(
\left \vert \mathbf{k}\right \vert +b_{1}+c_{1}\right)  \Gamma \left(  \left \vert
\mathbf{k}\right \vert +b_{2}+c_{2}\right)  }{\left(  \left(  2\left \vert
\mathbf{k}\right \vert +\left \vert \mathbf{b}\right \vert \right)
_{n-\left \vert \mathbf{k}\right \vert }\right)  ^{2}\left(  2n+\left \vert
\mathbf{b}\right \vert +\left \vert \mathbf{c}\right \vert -1\right)
\Gamma \left(  n+\left \vert \mathbf{k}\right \vert +\left \vert \mathbf{b}%
\right \vert +\left \vert \mathbf{c}\right \vert -1\right)  }\delta_{n,m}\\
  \times \prod \limits_{j=1}^{d}\frac{\left(  k_{j}!\right)  ^{2}\Gamma \left(
\left \vert \mathbf{k}^{j+1}\right \vert +2a_{1}+\frac{d-j}{2}\right)
\Gamma \left(  \left \vert \mathbf{k}^{j+1}\right \vert +2a_{2}+\frac{d-j}%
{2}\right)  }{2^{2\left \vert \mathbf{k}^{j+1}\right \vert }\left(  \left(
2\left \vert \mathbf{k}^{j+1}\right \vert +2\left \vert \mathbf{a}\right \vert
+d-j-1\right)  _{k_{j}}\right)  ^{2}}\delta_{k_{j},l_{j}}%
\end{multline*}
for $a_{1},a_{2},b_{1},b_{2},c_{1},c_{2}>0$ where $h_{\mathbf{k}}^{\left(
a_{1}+a_{2}-\frac{1}{2}\right)  }$ is given by (\ref{Norm}) and%
\begin{multline*}
 B_{n,\mathbf{k}}^{\left(  d+1\right)  }\left(
t,\boldsymbol{x};a_{1},a_{2},b_{1},b_{2},c_{1},c_{2}\right)  \\
=~_{3}F_{2}\left(  \left.
\genfrac{}{}{0pt}{0}{-n+\left \vert \mathbf{k}\right \vert ,\ n+\left \vert
\mathbf{k}\right \vert +\left \vert \mathbf{b}\right \vert +\left \vert
\mathbf{c}\right \vert -1,~\left \vert \mathbf{k}\right \vert +b_{1}-\frac{t}%
{2}}{2\left \vert \mathbf{k}\right \vert +\left \vert \mathbf{b}\right \vert
,~\left \vert \mathbf{k}\right \vert +b_{1}+c_{1}}%
~\right \vert ~1\right)  \\
  \times \left(  b_{1}-\frac{t}{2}\right)  _{\left \vert \mathbf{k}\right \vert
}\prod \limits_{j=1}^{d}\left \{  \Gamma \left(  a_{1}+\frac{\left \vert
\mathbf{k}^{j+1}\right \vert -x_{j}}{2}+\frac{d-j}{4}\right)  \Gamma \left(
a_{1}+\frac{\left \vert \mathbf{k}^{j+1}\right \vert +x_{j}}{2}+\frac{d-j}%
{4}\right)  \right.  \\
  \times \left.  _{3}F_{2}\left(  \left.
\genfrac{}{}{0pt}{0}{-k_{j},\ k_{j}+2\left(  \left \vert \mathbf{k}%
^{j+1}\right \vert +\left \vert \mathbf{a}\right \vert +\frac{d-j-1}{2}\right)
,\ a_{1}+\frac{\left \vert \mathbf{k}^{j+1}\right \vert +x_{j}}{2}+\frac{d-j}%
{4}}{\left \vert \mathbf{k}^{j+1}\right \vert +\left \vert \mathbf{a}\right \vert
+\frac{d-j}{2},\  \left \vert \mathbf{k}^{j+1}\right \vert +2a_{1}+\frac{d-j}{2}}%
%EndExpansion
~\right \vert ~1\right)  \right \}
\end{multline*}
or, in terms of Hahn polynomials (\ref{hahn})%
\begin{multline*}
\ B_{n,\mathbf{k}}^{\left(  d+1\right)  }\left(  t,\boldsymbol{x}%
;a_{1},a_{2},b_{1},b_{2},c_{1},c_{2}\right)  \\
=\dfrac{\left(  n-\left \vert
\mathbf{k}\right \vert \right)  !i^{\left \vert \mathbf{k}\right \vert -n}\left(
b_{1}-\frac{t}{2}\right)  _{\left \vert \mathbf{k}\right \vert }}{\left(
2\left \vert \mathbf{k}\right \vert +\left \vert \mathbf{b}\right \vert \right)
_{n-\left \vert \mathbf{k}\right \vert }\left(  \left \vert \mathbf{k}\right \vert
+b_{1}+c_{1}\right)  _{n-\left \vert \mathbf{k}\right \vert }}
 p_{n-\left \vert \mathbf{k}\right \vert }\left(  \dfrac{it}%
{2};~\left \vert \mathbf{k}\right \vert +b_{1},c_{2},\left \vert \mathbf{k}%
\right \vert +b_{2},c_{1}\right)  \\
\times%
{\displaystyle \prod \limits_{j=1}^{d}}
%EndExpansion
\left \{  \dfrac{k_{j}!i^{-k_{j}}}{\left(  \left \vert \mathbf{k}^{j+1}%
\right \vert +2a_{1}+\dfrac{d-j}{2}\right)  _{k_{j}}\left(  \left \vert
\mathbf{k}^{j+1}\right \vert +\left \vert \mathbf{a}\right \vert +\dfrac{d-j}%
{2}\right)  _{k_{j}}}\right.  \\
\times p_{k_{j}}\left(  -\dfrac{ix_{j}}{2};a_{1}+\dfrac{\left \vert
\mathbf{k}^{j+1}\right \vert }{2}+\dfrac{d-j}{4},a_{2}+\dfrac{\left \vert
\mathbf{k}^{j+1}\right \vert }{2}+\dfrac{d-j}{4}\right.  \\
\left.  \left.  ,a_{2}+\dfrac{\left \vert \mathbf{k}^{j+1}\right \vert }%
{2}+\dfrac{d-j}{4},a_{1}+\dfrac{\left \vert \mathbf{k}^{j+1}\right \vert }%
{2}+\dfrac{d-j}{4}\right)  \right.  \\
\times \left.  \Gamma \left(  a_{1}+\dfrac{\left \vert \mathbf{k}^{j+1}%
\right \vert -x_{j}}{2}+\dfrac{d-j}{4}\right)  \Gamma \left(  a_{1}%
+\dfrac{\left \vert \mathbf{k}^{j+1}\right \vert +x_{j}}{2}+\dfrac{d-j}%
{4}\right)  \right \},
\end{multline*}
for $d\geq1$.
\end{theorem}

\begin{proof}
The proof can be done by using the induction method in $d$. Let's start with the case $d=1$. In
this case, we get the specific function from (\ref{jac-spec-func})%
\begin{multline*}
g_{n,k_{1}}\left(  t,x_{1};a_{1},b_{1},c_{1},\beta_{1},\mu_{1},\gamma_{1}\right)    \\
=\left(  1-\tanh^{2}x_{1}\right)  ^{a_{1}}\left(  1+\tanh t\right)  ^{b_{1}}\left(  1-\tanh t\right)  ^{c_{1}}J_{k_{1},n}^{\beta_{1}%
,\mu_{1},\gamma_{1}}\left(  \tau_{1},\tau_{2}\right) \\
=\left(  1-\tanh^{2}x_{1}\right)  ^{a_{1}}\left(  1+\tanh t\right)^{b_{1}}\left(  1-\tanh t\right)  ^{c_{1}}\  \\
 \times P_{n-k_{1}}^{\left(  2k_{1}+2\mu_{1}+\beta_{1},\gamma_{1}\right)
}\left(  -\tanh t\right)  \left(  \frac{1+\tanh t}{2}\right)  ^{k_{1}}%
C_{k_{1}}^{\left(  \mu_{1}\right)  }\left(  \tanh x_{1}\right)
\end{multline*}
where $\tau_{1}=\frac{1+\tanh t}{2}$ and $\tau_{2}=\left(  \frac{1+\tanh t}%
{2}\right)  \tanh x_{1}.$ By use of Parseval's identity%
\begin{multline*}
  4\pi^{2} {\displaystyle \int \limits_{-\infty}^{\infty}} {\displaystyle \int \limits_{-\infty}^{\infty}} g_{n,k_{1}}\left(  t,x_{1};a_{1},b_{1},c_{1},\beta_{1},\mu_{1},\gamma
_{1}\right)  g_{m,l_{1}}\left(  t,x_{1};a_{2},b_{2},c_{2},\beta_{2},\mu
_{2},\gamma_{2}\right)  dx_{1}dt\\
  =\int \limits_{-\infty}^{\infty}\int \limits_{-\infty}^{\infty}%
\mathcal{F}%
\left(  g_{n,k_{1}}\left(  t,x_{1};a_{1},b_{1},c_{1},\beta_{1},\mu_{1}%
,\gamma_{1}\right)  \right)  \overline{%
\mathcal{F}%
\left(  g_{m,l_{1}}\left(  t,x_{1};a_{2},b_{2},c_{2},\beta_{2},\mu_{2}%
,\gamma_{2}\right)  \right)  }d\xi_{1}d\xi_{2},
\end{multline*}
it follows%
\begin{multline*}
4\pi^{2} {\displaystyle \int \limits_{-\infty}^{\infty}} {\displaystyle \int \limits_{-\infty}^{\infty}} \left(  1-\tanh^{2}x_{1}\right)  ^{a_{1}+a_{2}}\left(  1+\tanh t\right) ^{b_{1}+b_{2}}\left(  1-\tanh t\right)^{c_{1}+c_{2}} \\
\times J_{k_{1},n}^{\beta_{1},\mu_{1},\gamma_{1}}\left(  \tau_{1},\tau_{2}\right)  J_{l_{1},m}^{\beta_{2},\mu_{2},\gamma_{2}}\left(  \tau_{1},\tau_{2}\right)  dx_{1}dt\\
  =4\pi^{2}%
{\displaystyle \int \limits_{-\infty}^{\infty}}
{\displaystyle \int \limits_{-\infty}^{\infty}}
\left(  1-\tanh^{2}x_{1}\right)  ^{a_{1}+a_{2}}\left(  1+\tanh t\right)
^{b_{1}+b_{2}}\left(  1-\tanh t\right)  ^{c_{1}+c_{2}}\left(  \frac{1+\tanh
t}{2}\right)  ^{k_{1}+l_{1}}\\
 \times P_{n-k_{1}}^{\left(  2k_{1}+2\mu_{1}+\beta_{1},\gamma_{1}\right)
}\left(  -\tanh t\right)  P_{m-l_{1}}^{\left(  2l_{1}+2\mu_{2}+\beta
_{2},\gamma_{2}\right)  }\left(  -\tanh t\right)  \\
\times C_{k_{1}}^{\left(  \mu
_{1}\right)  }\left(  \tanh x_{1}\right)  C_{l_{1}}^{\left(  \mu_{2}\right)
}\left(  \tanh x_{1}\right)  dx_{1}dt \\
  =\pi^{2}2^{b_{1}+b_{2}+c_{1}+c_{2}+1}{\displaystyle \int \limits_{0}^{1}}
u^{b_{1}+b_{2}+k_{1}+l_{1}-1}\left(  1-u\right)  ^{c_{1}+c_{2}-1}P_{n-k_{1}%
}^{\left(  2k_{1}+2\mu_{1}+\beta_{1},\gamma_{1}\right)  }\left(  1-2u\right) \\ \times
P_{m-l_{1}}^{\left(  2l_{1}+2\mu_{2}+\beta_{2},\gamma_{2}\right)  }\left(
1-2u\right)  du
{\displaystyle \int \limits_{-1}^{1}}
\left(  1-v^{2}\right)  ^{a_{1}+a_{2}-1}C_{k_{1}}^{\left(  \mu_{1}\right)
}\left(  v\right)  C_{l_{1}}^{\left(  \mu_{2}\right)  }\left(  v\right)  dv
\end{multline*}
\begin{multline*}
  =\frac{2^{2a_{1}+2a_{2}+b_{1}+b_{2}+c_{1}+c_{2}-4}\left(  2\mu_{1}\right)
_{k_{1}}\left(  2\mu_{2}\right)  _{l_{1}}\left(  2k_{1}+2\mu_{1}+\beta
_{1}+1\right)  _{n-k_{1}}\left(  2l_{1}+2\mu_{2}+\beta_{2}+1\right)
_{m-l_{1}}}{\left(  n-k_{1}\right)  !\left(  m-l_{1}\right)  !k_{1}%
!l_{1}!\Gamma \left(  2a_{1}\right)  \Gamma \left(  2a_{2}\right)  \Gamma \left(
k_{1}+b_{1}+c_{1}\right)  \Gamma \left(  l_{1}+b_{2}+c_{2}\right)  }\\
  \times%
{\displaystyle \int \limits_{-\infty}^{\infty}}
{\displaystyle \int \limits_{-\infty}^{\infty}}
\text{ }\Gamma \left(  a_{1}+\frac{i\xi_{1}}{2}\right)  \Gamma \left(
a_{1}-\frac{i\xi_{1}}{2}\right)  \Gamma \left(  c_{1}+\frac{i\xi_{2}}%
{2}\right)  \Gamma \left(  b_{1}-\frac{i\xi_{2}}{2}\right)  ~\left(
b_{1}-\frac{i\xi_{2}}{2}\right)  _{k_{1}}\\
  \times \overline{\Gamma \left(  a_{2}+\frac{i\xi_{1}}{2}\right)
\Gamma \left(  a_{2}-\frac{i\xi_{1}}{2}\right)  }\overline{\Gamma \left(
c_{2}+\frac{i\xi_{2}}{2}\right)  \Gamma \left(  b_{2}-\frac{i\xi_{2}}%
{2}\right)  \left(  b_{2}-\frac{i\xi_{2}}{2}\right)  _{l_{1}}}\\
  \times~_{3}F_{2}\left(  \left.
\genfrac{}{}{0pt}{0}{-k_{1},\ k_{1}+2\mu_{1},\ a_{1}+\frac{i\xi_{1}}%
{2}}{2a_{1},\  \mu_{1}+1/2}%
~\right \vert ~1\right)  ~\overline{_{3}F_{2}\left(  \left.
\genfrac{}{}{0pt}{0}{-l_{1},\ l_{1}+2\mu_{2},\ a_{2}+\frac{i\xi_{1}}%
{2}}{2a_{2},\  \mu_{2}+1/2}%
~\right \vert ~1\right)  }\\
  \times~_{3}F_{2}\left(  \left.
\genfrac{}{}{0pt}{0}{-n+k_{1},\ n+k_{1}+b_{1}+c_{1}-1,~k_{1}+b_{1}-\frac
{i\xi_{2}}{2}}{2k_{1}+b_{1},~k_{1}+b_{1}+c_{1}}%
~\right \vert ~1\right)  ~\\
  \times \overline{_{3}F_{2}\left(  \left.
\genfrac{}{}{0pt}{0}{-m+l_{1},\ m+l_{1}+b_{2}+c_{2}-1,~l_{1}+b_{2}-\frac
{i\xi_{2}}{2}}{2l_{1}+b_{2},~l_{1}+b_{2}+c_{2}}%
~\right \vert ~1\right)  }d\xi_{1}d\xi_{2}.
\end{multline*}
By assuming
\begin{align*}
\mu_{1}  &  =\mu_{2}=a_{1}+a_{2}-\frac{1}{2}\\
\beta_{1}  &  =\beta_{2}=b_{1}+b_{2}-2a_{1}-2a_{2}\\
\gamma_{1}  &  =\gamma_{2}=c_{1}+c_{2}-1
\end{align*}
and considering the orthogonality relations of (\ref{ort}) and (\ref{ort-J}), it is seen that the special function%
\begin{multline*}
\ B_{n,k_{1}}^{\left(  2\right)  }\left(  t,x_{1};a_{1},a_{2},b_{1}%
,b_{2},c_{1},c_{2}\right)    =\  \Gamma \left(  a_{1}-\frac{x_{1}}{2}\right)
\Gamma \left(  a_{1}+\frac{x_{1}}{2}\right)  \left(  b_{1}-\frac{t}{2}\right)
_{k_{1}}\\
  \times~_{3}F_{2}\left(  \left.
\genfrac{}{}{0pt}{0}{-k_{1},\ k_{1}+2\left(  a_{1}+a_{2}\right)
-1,\ a_{1}+\frac{x_{1}}{2}}{a_{1}+a_{2},\ 2a_{1}}%
%EndExpansion
~\right \vert ~1\right) \\
 \times~_{3}F_{2}\left(  \left.
\genfrac{}{}{0pt}{0}{-n+k_{1},\ n+k_{1}+b_{1}+b_{2}+c_{1}+c_{2}-1,~k_{1}%
+b_{1}-\frac{t}{2}}{2k_{1}+b_{1}+b_{2},~k_{1}+b_{1}+c_{1}}%
~\right \vert ~1\right)
\end{multline*}
namely,%
\begin{multline*}
 B_{n,k_{1}}^{\left(  2\right)  }\left(  t,x_{1};a_{1},a_{2},b_{1}%
,b_{2},c_{1},c_{2}\right)     \\ =\frac{\left(  n-k_{1}\right)  !k_{1}!i^{-n}%
}{\left(  2a_{1}\right)  _{k_{1}}\left(  a_{1}+a_{2}\right)  _{k_{1}}\left(
2k_{1}+b_{1}+b_{2}\right)  _{n-k_{1}}\left(  k_{1}+b_{1}+c_{1}\right)
_{n-k_{1}}}\\
  \times~\Gamma \left(  a_{1}-\frac{x_{1}}{2}\right)  \Gamma \left(
a_{1}+\frac{x_{1}}{2}\right)  \left(  b_{1}-\frac{t}{2}\right)  _{k_{1}}\\
  \times p_{k_{1}}\left(  \frac{-ix_{1}}{2};a_{1},a_{2},a_{2},a_{1}\right)
p_{n-k_{1}}\left(  \frac{it}{2};b_{1}+k_{1},c_{2},b_{2}+k_{1},c_{1}\right)
\end{multline*}
satisfies the relation%
\begin{multline*}
  \int \limits_{-\infty}^{\infty}\int \limits_{-\infty}^{\infty}\Gamma \left(
b_{1}-\frac{it}{2}\right)  \Gamma \left(  b_{2}+\frac{it}{2}\right)
\Gamma \left(  c_{1}+\frac{it}{2}\right)  \Gamma \left(  c_{2}-\frac{it}%
{2}\right) \\
\times  \ B_{n,k_{1}}^{\left(  2\right)  }\left(  it,ix_{1};a_{1},a_{2},b_{1}%
,b_{2},c_{1},c_{2}\right)  \ B_{m,l_{1}}^{\left(  2\right)  }\left(
-it,-ix_{1};a_{2},a_{1},b_{2},b_{1},c_{2},c_{1}\right)  dx_{1}dt\\
  =\frac{\pi^{2}2^{-2a_{1}-2a_{2}+5}h_{k_{1}}^{\left(  a_{1}+a_{2}-\frac
{1}{2}\right)  }\Gamma \left(  n+k_{1}+b_{1}+b_{2}\right)  \Gamma \left(
n-k_{1}+c_{1}+c_{2}\right)  }{\left(  \left(  2k_{1}+b_{1}+b_{2}\right)
_{n-k_{1}}\right)  ^{2}\left(  \left(  2a_{1}+2a_{2}-1\right)  _{k_{1}%
}\right)  ^{2}}\\
 \times \frac{\left(  k_{1}!\right)  ^{2}\left(  n-k_{1}\right)
!\Gamma \left(  k_{1}+b_{1}+c_{1}\right)  \Gamma \left(  k_{1}+b_{2}%
+c_{2}\right)  \Gamma \left(  2a_{1}\right)  \Gamma \left(  2a_{2}\right)
}{\left(  2n+b_{1}+b_{2}+c_{1}+c_{2}-1\right)  \Gamma \left(  n+k_{1}%
+b_{1}+b_{2}+c_{1}+c_{2}-1\right)  }\delta_{k_{1},l_{1}}\delta_{n,m}%
\end{multline*}
where $h_{k_{1}}^{\left(  a_{1}+a_{2}-\frac{1}{2}\right)  }$ is given by
(\ref{gnorm}). Similarly, if we substitute (\ref{jac-spec-func}) and
(\ref{Jac-Fourier}) in the Parseval's identity (\ref{multi}), the necessary
calculations give the desired result.
\end{proof}

\textbf{Authors' contributions}
Both authors contributed equally to this work. Both authors have read and approved the final manuscript.\\

\textbf{Funding} The work of the second author has been partially supported by a research grant of the Agencia Estatal de Investigaci\'on, Spain, Grant PID2020- 113275GB-I00 funded by MCIN/AEI/10.13039/501100011033 and by ``ERDF A way of making Europe'', by the ``European Union'', as well as by grant ED431B 2024/42 (Xunta de Galicia). \\

\textbf{Data availability}
Data sharing is not applicable to this article as no data sets were generated or analyzed during the current study.\\

\section*{Declarations}
\textbf{Conflict of interest} The authors declare no competing interests.\\

\textbf{Ethical Approval} Not applicable.

\end{document}